\documentclass[11pt, reqno]{article}

\usepackage[utf8]{inputenc}
\usepackage[T1]{fontenc}

\usepackage{setspace}
\usepackage[margin=1.25in]{geometry}
\usepackage{parskip}
\usepackage[sc]{mathpazo}
\usepackage[T1]{eulervm}
\usepackage[scaled]{helvet}

\usepackage{amsmath,amsthm, amssymb, amsfonts}
\usepackage{mathtools}
\usepackage{graphicx}
\usepackage[hidelinks]{hyperref}
\usepackage{xcolor}

\newcommand{\arxiv}[1]{\href{https://arxiv.org/abs/#1}{\texttt{ArXiv:#1}}}
\newcommand{\arxivmath}[1]{\href{https://arxiv.org/abs/math/#1}{\texttt{ArXiv:#1}}}
\newcommand{\arxivmaph}[1]{\href{https://arxiv.org/abs/math-ph/#1}{\texttt{ArXiv:#1}}}
\newcommand{\arxivcnma}[1]{\href{https://arxiv.org/abs/cond-mat/#1}{\texttt{ArXiv:#1}}}

\theoremstyle{plain}
\newtheorem{thm}{Theorem}
\newtheorem{cor}{Corollary}[section]
\newtheorem{lem}{Lemma}[section]
\newtheorem{prop}{Proposition}[section]

\theoremstyle{definition}

\numberwithin{equation}{section}

\newcommand{\pr}[1]{\mathbf{Pr}\left[#1\right]}
\newcommand{\ind}[1]{\mathbf{1}_{\{ #1 \}}}
\newcommand{\dt}[1]{\mathrm{det}\left(#1\right)}
\newcommand{\intz}[1]{\frac{1}{2 \pi \mathbold{i}}\oint \limits_{|z|= #1}}

\newcommand{\D}{\Delta}
\newcommand{\dr}{\nabla}
\newcommand{\vt}{\, | \,}

\newcommand{\eps}{\varepsilon}
\newcommand{\Z}{\mathbb{Z}}
\newcommand{\R}{\mathbb{R}}
\newcommand{\C}{\mathbb{C}}
\newcommand{\W}{\mathbb{W}}
\newcommand{\Ai}{\mathrm{Ai}}

\newcommand{\G}{\mathcal{G}}
\newcommand{\Gb}{\mathbold{G}}
\newcommand{\J}{\mathcal{J}}
\newcommand{\K}{\mathcal{K}}

\title{On inhomogeneous polynuclear growth}

\author{Kurt Johansson\thanks{\textsc{Department of Mathematics, KTH Royal Institute of Technology}. \textit{Email}: \texttt{kurtj@kth.se}}
\and
Mustazee Rahman\thanks{\textsc{Department of Mathematical Sciences, Durham University}. \textit{Email}: \texttt{mustazee@gmail.com}}}

\date{}

\begin{document}
\maketitle

\setcounter{tocdepth}{2}
\tableofcontents

\begin{abstract}
\noindent	This article studies the inhomogeneous geometric polynuclear growth model, the
	distribution of which is related to Schur functions. We explain a method to derive its distribution functions
	in both space-like and time-like directions, focusing on the two-time distribution. Asymptotics of the two-time
	distribution in the KPZ-scaling limit is then considered, extending to two times several single-time
	distributions in the KPZ universality class.
\end{abstract}
\medskip
{\footnotesize
\emph{Keywords}: KPZ universality, last passage percolation, polynuclear growth, two-time distribution \\
\emph{AMS 2020 subject classification}: Primary 60F05, 60K35, 82C23, 82C24; Secondary 05E10, 15A15, 30E20.
}
\newpage

\section{Introduction} \label{sec:1}
This article looks at the inhomogeneous polynuclear growth model, known also as geometric directed last passage percolation.
It is defined in terms of parameters $ a_i, b_j \in [0,1]$ with $0 < a_i b_j < 1$ for every $i,j \geq 1$. Suppose
$\omega_{i,j} \sim \mathrm{Geom}(a_i b_j)$ are geometric random variables of rate $a_i b_j$, that is,
$$ \pr{\omega_{i,j} = k} = (1- a_ib_j) (a_i b_j)^k \quad k = 0,1,2,\ldots.$$
The growth function associated to this random environment is
\begin{equation} \label{eqn:G} \Gb(m,n) = \max \, \{ \Gb(m-1,n) , \Gb(m,n-1)\} + \omega(m,n).\end{equation}
The boundary conditions are $\Gb(m,0) \equiv 0$ and $\Gb(0,n) = x_n$ for $x_1 \leq x_2 \leq x_3 \leq \cdots$.

Function $\Gb$ models up/right growth in a quadrant. Rotating the quadrant by 45 degrees,
$\Gb$ can be thought of as a height interface $H(x,t)$ that grows with time. Specifically, if one defines
\begin{equation} \label{eqn:pngheight}
H(x,t) = \Gb \left(\frac{t+x+1}{2},\frac{t-x+1}{2} \right )
\end{equation}
for odd $x+t$ with $|x| < t$, and extends $H$ to $x \in \R$ by linear interpolation, then $H(x,t)$ represents
the height above $x$ of an interface at time $t$. It is in this setting that the term polynuclear growth model is used; see \cite{KS, PS}.
This and closely related inhomogeneous growth models have been looked at in \cite{BBP, BR, BDR, GTW, IS, JoDPG}.

Definition \eqref{eqn:G} also leads to the expression
$$\Gb(m,n) = \max_{\pi} \, \sum_{(i,j) \in \pi} \omega(i,j)$$
where the maximum is over all up/right lattice paths $\pi$ from $(1,1)$ to $(m,n)$.
Up/right means that the paths move in the direction (0,1) or (1,0) at each step.
In this way $\Gb(m,n)$ represents the last passage time among directed paths from $(1,1)$ to $(m,n)$
with respect to the weights $\omega(i,j)$. The behaviour of $\Gb$ in this setting, with inhomogeneity,
especially its macroscopic shape, has been recently studied in \cite{Em, EJS}.

From a different viewpoint $\Gb$ defines a totally asymmetric exclusion process where particle $j$ makes
its $i$-th jump at geometric rate $a_ib_j$. In yet another it defines a measure on Young diagrams in terms of Schur functions.
We do not to explore these here directly, although to do so would be interesting. See the papers \cite{DW, KPS, RS}
for recent works on such particle systems with inhomogeneous hopping rates, and \cite{As, BP, Josurvey, Ok}
for some examples of the relation to Young diagrams.

Our interest lies in various aspects of the inhomogeneous growth model, such as its transition probability
as a Markov chain, its distribution function along certain space-like paths, and its two-time distribution in both
the discrete and asymptotic KPZ-scaling limit. We show how to compute all of these in terms of determinants.
These results are expanded on in the coming sections. The KPZ-scaling limit leads to two-time distributions that
extend previously known single-time distributions in the KPZ universality class. We also provide formulas
for the inhomogeneous exponential last passage percolation model as a limiting case of the geometric one.

This article grew from our earlier work \cite{JR} in an effort to extend the results from the homogeneous to the inhomogeneous setting.
The inhomogeneity makes the model more challenging, more general and perhaps more applicable. It turns out that with
modifications, the basic ideas in \cite{JR} can be extended and many aspects of the inhomogeneous model can be understood.
The remainder of the introduction describes these results.

\paragraph{\textbf{Markovian transition probability.}}

Consider the vector-valued process
\begin{equation} \label{eqn:Gvec} \vec{\Gb}(m) = (\Gb(m,1), \Gb(m,2), \ldots, \Gb(m,N)),\end{equation}
which, for a given $N$, is an inhomogeneous Markov chain taking values in
$$\W_N = \{ z \in \Z^N: z_1 \leq z_2 \leq \cdots \leq z_N\}.$$
Its transition probabilities are given by a determinant.

\begin{thm} \label{thm:1}
	For $x, y \in \W_N$,
	$$ \pr{\vec{\Gb}(m) = y \vt \vec{\Gb}(0) = x} = \dt{M(i,j \vt x,y)}_{i,j},$$
	where
	\begin{equation*}
	M(i,j \vt x,y) = \frac{1}{2\pi \mathbold{i}} \oint_{|z| = R} \frac{dz}{z}\,
	\frac{(zb_j)^{y_j}}{(zb_i)^{x_i}} \frac{\prod_{k=1}^j (z - 1/b_k)}{\prod_{k=1}^i (z-1/b_k)} \prod_{k=1}^m \frac{1-a_kb_j}{1- a_k/z}.
	\end{equation*}
	The radius $R > \max \{1/b_1, \ldots, 1/b_N \}$ and the circular contour is oriented counter-clockwise.
\end{thm}
The transition probabilities for the model \eqref{eqn:G} with exponential passage times $\omega_{i,j}$ are obtained in $\S$\ref{sec:5}.
Theorem \ref{thm:1} generalizes results from \cite{DW, RS} about transition probabilities of certain particle systems with
inhomogeneous jump rates.

\paragraph{\textbf{Multi-spatial distribution at a single time.}}

Suppose $n_1 < n_2 < \cdots < n_p = N$ and consider the distribution function
\begin{equation} \label{eqn:singletime}
\pr{\Gb(m,n_1) < h_1, \Gb(m,n_2) < h_2, \ldots, \Gb(m,n_p) < h_p \vt \vec{\Gb}(0)=x}\,.
\end{equation}
This can be thought of as the distribution function of $\Gb$ along a space-like path as explained below; see also \cite{BFS}.
For it one has the following formula.

\begin{thm} \label{thm:2}
For $x \in \W_N$ and $0 = n_0 < n_1 < n_2 < \cdots < n_p = N$, set $h(j) = h_k$ for every $j \in (n_{k-1}, n_k]$. Then,
$$\pr{\Gb(m,n_1) < h_1, \Gb(m,n_2) < h_2, \ldots, \Gb(m,n_p) < h_p \vt \vec{\Gb}(0) = x} = \dt{F(i,j \vt x)}_{i,j},$$
where $F$ is the $N \times N$ matrix
\begin{equation*}
F(i,j \vt x) = \frac{1}{2\pi \mathbold{i}} \oint_{|z| = R} dz\, \frac{(zb_j)^{h(j)-1}}{(zb_i)^{x_i}}
\frac{\prod_{k=1}^{j-1} (z - 1/b_k)}{\prod_{k=1}^i (z-1/b_k)}  \prod_{k=1}^m \frac{1-a_kb_j}{1-a_k/z}.
\end{equation*}
The radius $R > \max \{1/b_1, \ldots, 1/b_N \}$.
\end{thm}
In $\S$\ref{sec:fredholm} we explain how this determinant can be expressed as a Fredholm determinant,
which is often better for extracting asymptotics. Proposition \ref{prop:spatialFredholm} presents the Fredholm
determinantal formula when $\vec{\Gb}(0) = 0$. In $\S$\ref{sec:5} a similar formula is presented for the exponential model.

The distribution function \eqref{eqn:singletime} may be used to study the asymptotic single time, multi-spatial
distribution function of the height interface $H(x,t)$ from \eqref{eqn:pngheight} under the KPZ scaling limit.
Indeed, when $m = T$ and $n_k = T - 2x_k T^{2/3} -1$, \eqref{eqn:singletime} provides the joint distribution
of $x_k \mapsto H(T - x_kT^{2/3}, x_kT^{2/3}+1)$. For large values of $T$, the slow de-correlation phenomenon (see \cite{CFPb})
implies that $H(T - x_kT^{2/3}, x_kT^{2/3}+1) - H(T, x_kT^{2/3})$ is of order $o_p(T^{1/3})$ for any
finite number of $x_k$s. So the asymptotic finite dimensional distributions of the function $x \mapsto H(T, xT^{2/3})/ T^{1/3}$,
which is its KPZ-scaling limit (see \cite{KPZ}), can be obtained from the distribution function \eqref{eqn:singletime}.
We will, however, not discuss how to use Theorem \ref{thm:2} to analyze these asymptotics in an inhomogeneous model.

\paragraph{\textbf{The two-time distribution.}}

Of much interest is the two-time distribution of $\Gb$:
\begin{equation} \label{eqn:2timedistribution}
\pr{ \Gb(m,n) < h,  \Gb(M,N) < H \vt \vec{\Gb}(0)=x}
\end{equation}
for $m < M$ and $n < N$. In terms of the height interface \eqref{eqn:pngheight} this is the joint distribution
of the interface at two different times, namely at $t_1 = m+n-1$ and $t_2 = M+N-1$.

\begin{thm} \label{thm:twotime}
	The two-time distribution function \eqref{eqn:2timedistribution} equals
	$$ \pr{\Gb(m,n) < h, \Gb(M,N) < H \vt \vec{\Gb}(0)=x} =
	\frac{1}{2 \pi \mathbold{i}} \oint \limits_{|\theta| = r} d \theta\, \frac{\dt{\theta^{\ind{i > n}}L_1 - \theta^{-\ind{i \leq n}}L_2}}{\theta -1}$$
	where the radius $r > 1$ and $L_1, L_2$ are the following $N \times N$ matrices.
	\begin{align*}
	L_1(i,j) &= \frac{1}{(2 \pi \mathbold{i})^2} \oint \limits_{|z|=R_1} dz \oint \limits_{|w|=R_2} dw\,
	\frac{(b_iz)^{h-1} (b_jw)^{H-h}}{(b_iz)^{x_i} (z-w)} \frac{\prod_{k=1}^n(z-1/b_k)}{\prod_{k=1}^i(z-1/b_k)} \frac{\prod_{k=1}^{j-1}(w-1/b_k)}{\prod_{k=1}^n(w-1/b_k)} \\
	& \times \prod_{k=1}^m \frac{1-a_kb_i}{1-a_k/z} \prod_{k= m+1}^M \frac{1-a_kb_j}{1-a_k/w}.
	\end{align*}
	The contours are arranged so that $R_1 > R_2 > \max_k \{1/b_k\}$.
	The matrix $L_2$ looks the same except that the ordering of the contours is reversed to $R_2 > R_1 > \max \{1/b_k\}$.
\end{thm}
A formula for the model with exponential passage times is given in $\S$\ref{sec:5}.
In $\S$\ref{sec:6} we explain how the determinant above can be expressed as a Fredholm
determinant, and carry out the procedure to get a Fredholm determinant formula when
$\vec{\Gb}(0) = 0$ -- see Theorem \ref{thm:4}.

\paragraph{\textbf{Asymptotic considerations.}}
Theorem \ref{thm:4} leads us to investigate asymptotics of the two-time distribution in two cases.
The first is a perturbation of the homogeneous model where $a_1, a_2, \ldots, a_r$ are variable,
all other $a_i$s are set to be $q \in (0,1)$, and all the $b_j$s are 1. This model is studied in $\S$\ref{sec:7}.
The asymptotics are according to KPZ-scaling, \cite{KPZ}, whereby the parameters $n, N, m, M, h$ and $H$
are scaled as in \eqref{kpzscaling}. In order to get a meaningful limit the $a_k$s need to be scaled accordingly as well,
in the form (see \eqref{ascaling})
$$a_k = \sqrt{q} - c_{q} \lambda_k T^{-1/3}$$
for a $q$-dependent constant $c_{q}$, parameters $\lambda_k > 0$ and $T$ the large scaling parameter.

The asymptotics lead to a determinantal formula that is the two-time analogue of
the Baik-Ben\,Arous-P\'ech\'e distribution, which appears as the asymptotic largest eigenvalue
distribution of finite rank perturbations of complex Wishart matrices \cite{BBP}. Our formula is presented
in Theorem \ref{thm:5} and the result is stated in $\S$\ref{sec:7} (it needs introducing notation).
With more effort these methods should extend to formulas for the entire multi-time distribution of this model.
In the limit $\lambda_k \to +\infty$ one recovers the two-time distribution of the homogeneous model \cite{JoTwo}.
In the other limit $\lambda_k \to -\infty$ the distribution is known to have a Gaussian law \cite{CFPa}.
It is an intriguing question if this distribution also arises from a random matrix model.

The second case of asymptotics, studied in $\S$\ref{sec:8}, is the two-time distribution when
$a_1 = \sqrt{q}$, $a_i = q$ for $i > 1$ and every $b_j = 1$.
In other words, the weights $\omega(i,1)$ along the bottom row are distributed as $\mathrm{Geom}(\sqrt{q})$
while the rest are $\mathrm{Geom}(q)$. In this case a geodesic path $\pi$, that which attains value $\Gb(n,n)$,
spends an order of $n^{2/3}$ steps on the bottom row before venturing upwards to $(n,n)$. In terms of the
height interface $H(x,t)$ from \eqref{eqn:pngheight}, in the KPZ-scaling limit, this leads to a limiting interface
$\mathbold{H}(x,t)$ that starts at time 0 as a one-sided Brownian motion:
\begin{equation*} 
\mathbold{H}(x,0) = \begin{cases}
\sqrt{2} \, B(x) & \text{for}\; x \geq 0 \\
-\infty & \text{for}\; x < 0
\end{cases}.
\end{equation*}

Here $B(x)$ is standard Brownian motion. This is because under KPZ-scaling the contribution of the
weights $\omega(i,j)$ to $\Gb(n,n)$ along the bottom row, which is a random walk, scales to a Brownian motion.
The existence of the limit interface $\mathbold{H}(x,t)$ follows from results about the
KPZ fixed point in \cite{MQR}. See also \cite{BR, IS} for more on this model.
Theorem \ref{thm:6} gives the two-time distribution
\begin{equation} \label{eqn:Hstat}
\pr{ \mathbold{H}(x_1,t_1) < \xi_1, \mathbold{H}(x_2,t_2) < \xi_2}
\end{equation}
of this interface.

To conclude this introduction we remark that many aspects of the two-time distribution of $\Gb$ in the homogeneous setting,
and more generally its multi-time distribution, have been studied recently. Limit theorems have been established
in \cite{BL, DOV, NaDo, Liu, MQR, JoTwo, JR} and the two-time correlation function has been investigated in \cite{BaGa, NaDoTa, FO}.
See also the surveys \cite{BoGo, CoKPZ, QuKPZ} for general introduction to growth models in the KPZ universality class.
Related works can be found within these references as well.

\section{Computations for the inhomogeneous model} \label{partI}

\subsection{Markovian transition probability} \label{sec:2}
Theorem \ref{thm:1} will be proven by induction, computing the $m$-th step transition matrix as the convolution of the $(m-1)$-th step and 1-step
transition matrices. We begin with the 1-step transition matrix, which has been derived in terms of symmetric functions in \cite{DW}.

\subsubsection{The 1--step transition matrix}

Let $h_{\ell}(\alpha)$ be the $\ell$-th complete homogeneous symmetric polynomial  in variables $\alpha = (\alpha_1, \alpha_2, \ldots, \alpha_N)$.
Set $h_0 = 1$ and $h_{\ell} = 0$ if $\ell < 0$. Recall that
$$h_{\ell}(\alpha) = \sum_{\substack{(k_1,\ldots, k_N) \\ k_1 + \cdots k_N = \ell \\ k_i \geq 0}} \alpha_1^{k_1} \cdots \alpha_N^{k_N}.$$
Write $\alpha^{(i,j)} = (0, \ldots, 0, \alpha_{i+1}, \ldots, \alpha_j, 0,\ldots, 0)$ for $1 \leq i < j \leq N$, where the vector has $N$ components.
Set $h_{\ell}^{(i,j)}(\alpha) = h_{\ell}(\alpha^{(i,j)})$ and $h^{(i,i)}_{\ell}(\alpha) = \ind{\ell=0}$.

We use the same notation for the $\ell$-th elementary symmetric polynomials $e_{\ell}(\alpha)$:
$$ e_{\ell}(\alpha) = \sum_{\substack{ S \subset [N] \\ |S| = \ell}} \prod_{i \in S} \alpha_i .$$

Define, for $k \in \Z$, the functions
\begin{align*}
& w_{\alpha}^{(i,j)}(k) = \begin{cases}
\sum_{\ell=0}^{j-i} (-1)^{\ell} e_{\ell}^{(i,j)}(\alpha) \ind{k \geq \ell} & j \geq i \\
\sum_{\ell = 0}^{\infty} h_{\ell}^{(j,i)}(\alpha) \ind{k \geq \ell} & j < i
\end{cases}
\end{align*}

Now suppose $0 < p_j < 1$ for $1 \leq j \leq N$ and consider $\omega_j \sim \mathrm{Geom}(p_j)$.
For $x \in \W_N$, define $\vec{\Gb}(0) = x$ and $\vec{\Gb}(1) = (\Gb(1), \ldots, \Gb(N))$ according to
$$ \Gb(j) = \max \{ \Gb(j-1), x_j \} + \omega_j, \quad \Gb(0) =-\infty.$$
Set $1/p = (1/p_1, \ldots, 1/p_N)$. Then for $y \in \W_N$,
\begin{equation} \label{eqn:DW}
\pr{\vec{\Gb}(1) = y \vt \vec{\Gb}(0) = x} = \prod_{k=1}^N (1-p_k)p_k^{y_k-x_k}\, \dt{w^{(i,j)}_{1/p}(y_j+j - x_i-i)}_{i,j}.
\end{equation}
This is proved in \cite[Theorem 1]{DW} by using the RSK algorithm and certain intertwining between Markov kernels.
It was proved for the homogeneous model, where every $p_k=p$, in \cite{JoMar} by induction, and earlier
for a Brownian last passage model in \cite{Wa}.

Let us express \eqref{eqn:DW} in terms of contour integrals, which will be more suitable for our purposes. We have
\begin{align*}
\sum_{\ell \in \Z} e_{\ell}(\alpha) z^{\ell} & = \prod_{k=1}^N (1 + \alpha_k z) \quad z \in \C, \\
\sum_{\ell \in \Z} h_{\ell}(\alpha) z^{\ell} & = \prod_{k=1}^N (1-\alpha_k z)^{-1} \quad |z| < \min_k \{1/\alpha_k\}.
\end{align*}
By substituting in $\alpha_k = 0$ for $k \leq i$ and $k > j$, we infer that
$$ \sum_{\ell \in \Z} e_{\ell}^{(i,j)}(\alpha)z^{\ell} = \prod_{k=i+1}^j (1+\alpha_k z) \quad \text{and} \quad 
\sum_{\ell \in \Z} h_{\ell}^{(i,j)}(\alpha)z^{\ell} = \prod_{k=i+1}^j (1-\alpha_k z)^{-1}.$$
Observe also that
$$ \ind{k \geq \ell} = \frac{1}{2 \pi \mathbold{i}} \oint \limits_{|z| = r} dz\, \frac{1}{z^{k-\ell+1}(1-z)} \quad \text{for}\; r < 1.$$
The circular contour $\{|z| = r\}$ is oriented counter-clockwise.

Using these representations, we find that for $j \geq i$,
\begin{align*}
\sum_{\ell = 0}^{j-i} (-1)^{\ell} e_{\ell}^{(i,j)}(\alpha) \ind{k \geq \ell} & =
\intz{r} \frac{dz}{z(1-z)} \sum_{\ell=0}^{j-i} (-1)^{\ell} e_{\ell}^{(i,j)}(\alpha) z^{\ell-k} \\
& = \intz{r} \frac{dz}{z^{k+1}(1-z)} \sum_{\ell=0}^{j-i} e_{\ell}^{(i,j)}(\alpha)(-z)^{\ell}\\
& = \intz{r} dz\, \frac{\prod_{\ell=i+1}^j (1-\alpha_{\ell}z)}{z^{k+1}(1-z)} \quad (r < 1).
\end{align*}
Likewise, for $j \leq i$,
$$\sum_{\ell=0}^{\infty} h_{\ell}^{(j,i)}(\alpha)\ind{k \geq \ell} = \intz{r} dz\, \frac{\prod_{\ell={j+1}}^i \, (1-\alpha_{\ell}z)^{-1}}{z^{k+1}(1-z)},$$
where $r < \min_k \{1/\alpha_k\}$. It follows from these identities that
$$w_{1/p}^{(i,j)}(k) = \intz{r} dz\, \frac{\prod_{k=1}^j (1- z/p_k)}{\prod_{k=1}^{i}(1-z/p_k)} \cdot \frac{1}{z^{k+1}(1-z)},$$
provided that $r < \min\{p_1, \ldots, p_N\}$.

Therefore, \eqref{eqn:DW} now reads as
$$\pr{\vec{\Gb}(1) = y \vt \vec{\Gb}(0) = x} = \prod_{k=1}^N (1-p_k)p_k^{y_k-x_k}\,
\dt {\intz{r} \frac{dz}{z(1-z)} \frac{z^{x_i+i}}{z^{y_j+j}} \frac{\prod_{k=1}^j (1-z/p_k)}{\prod_{k=1}^i(1-z/p_k)}}.$$
Push the factors of $p_j^{y_j}$ and $(1-p_j)$ into the $j$-th column of the determinant, and $p_i^{-x_i}$ into the $i$-th row.
Then, expressing
$$ z^{-y_j-j} \prod_{k=1}^j (1-z/p_k) = z^{-y_j} \prod_{k=1}^j (z^{-1}- p_k^{-1}) \;\;\text{and}\;\;
z^{x_i+i} \prod_{k=1}^{i} (1-z/p_k)^{-1} = z^{x_i} \prod_{k=1}^{i}(z^{-1}-p_k^{-1})^{-1}$$
imply that
$$ \pr{\vec{\Gb}(1) = y \vt \vec{\Gb}(0) = x} = \dt{A(i,j)}$$
where
$$ A(i,j) = \intz{r} \frac{dz}{z} \frac{(z/p_i)^{x_i}}{(z/p_j)^{y_j}} \,\frac{\prod_{k=1}^j (z^{-1}-p_k^{-1})}{\prod_{k=1}^i (z^{-1}-p_k^{-1})} \,\frac{1-p_j}{1-z}.$$
Changing variables $z \mapsto z^{-1}$ gives
$$A(i,j) = \intz{R} \frac{dz}{z} \frac{(zp_j)^{y_j}}{(zp_i)^{x_i}} \, \frac{\prod_{k=1}^j (z-p_k^{-1})}{\prod_{k=1}^i (z-p_k^{-1})} \, \frac{1-p_j}{1-1/z},$$
with $R > \max \{1/p_1, \ldots, 1/p_N\}$.

The 1-step transition matrix of the inhomogeneous growth model is obtained by having $p_k = a_1 b_k$
and changing variables $z \mapsto z/a_1$ in the integral above.
(One encounters a conjugation factor of $a_1^{i-j}$ but this does not affect the determinant.)

\subsubsection{The $m$--step transition matrix}
Assuming the form of the $(m-1)$-step transition probability of $\Gb$ given by Theorem \ref{thm:1}, we convolve it with the 1-step transition
probability from the previous section to obtain the $m$-step transition matrix. The $(m-1)$-step transition matrix has the form
\begin{align*}
Q(x,y) & = \prod_{j=1}^N \prod_{k=1}^{m-1} (1-a_kb_j) \dt{B(i,j)} \\
B(i,j) & = \intz{R} \frac{dz}{z}\, \frac{(zb_j)^{y_j}}{(zb_i)^{x_i}} \frac{\prod_{k=1}^j(z- 1/b_k)}{\prod_{k=1}^i (z- 1/b_k)} a(z) \\
a(z) &= \prod_{k=1}^{m-1} \frac{1}{1- a_k/z}.
\end{align*}
The function $a(z)$ is bounded and analytic outside the disk  $\{|z| > \max_k a_k \}$,
and we have suppressed the dependence of $B$ on $x$ and $y$.

Write $p_k = a_m b_k$ and denote $P(x,y)$ the 1-step transition matrix associated to $p_1, \ldots, p_N$ from the previous section. Then,
$$\pr{\vec{\Gb}(m) = y \vt \vec{\Gb}(0) = x} = \sum_{u \in \W_N} Q(x,u) P(u,y).$$
The following proposition establishes Theorem \ref{thm:1}.
\begin{prop}\label{prop:1}
Fix $x, y \in \W_N$. Let $Q$ have the form above and $P$ be the 1-step transition matrix associated to $p_k = a_m b_k$.
Then,
$$\sum_{u \in \W_N} Q(x,u) P(u,y) = \prod_{j=1}^N \prod_{k=1}^m (1-a_kb_j) \dt{K(i,j)}$$
where
$$K(i,j) = \intz{R} \frac{dz}{z}\, \frac{(zb_j)^{y_j}}{(zb_i)^{x_i}} \frac{\prod_{k=1}^j(z- 1/b_k)}{\prod_{k=1}^i (z- 1/b_k)} \frac{a(z)}{1-a_m/z}.$$
\end{prop}
We will spend the rest of this section proving this proposition.

Define the functions $f_{i,j}(u)$ and $\widetilde{g}_{i,j}(u)$, for $u \in \Z$, by
\begin{align*}
f_{i,j}(u) &= \intz{R_1} \frac{dz}{z}\, \frac{z^u}{(zb_i)^{x_i}} \frac{\prod_{k=1}^j(z- 1/b_k)}{\prod_{k=1}^i (z- 1/b_k)} \, a(z), \quad R_1 > \max_k \{1/b_k\}\, ;\\
\widetilde{g}_{i,j}(u) &= \frac{1}{2\pi \mathbold{i}} \oint \limits_{|w|=R_2} dw \, (a_m w)^u (wp_j)^{y_j} \frac{\prod_{k=1}^j(w- 1/p_k)}{\prod_{k=1}^i (w- 1/p_k)} \,
\frac{1}{w-1},\quad R_2 > \max_k \{1/p_k \}.
\end{align*}
It will be useful later to note that $f_{i,j}$ and $\widetilde{g}_{i,j}$ vanish for sufficiently negative values of $u$,
namely $f_{i,j}(u) = 0$ for $u \leq x_1-N$ and $\widetilde{g}_{i,j}(u) = 0$ for $u \leq - (y_N + N)$. This is
because the integrands then decay at least to the order $|z|^{-2}$ and so the contours can be contracted to infinity.

By factoring out $b_j^{u_j}$ from the columns of $B(i,j)$ and $b_i^{-u_i}$ from the rows of $A(i,j)$, and using that $p_i/b_i \equiv a_m$,
we see that $\dt{B} = \prod_k b_k^{u_k} \dt{f_{i,j}(u_j)}$ and $\dt{A} = \prod_k b_k^{-u_k} (1-p_k) \dt{\widetilde{g}_{i,j}(-u_i)}$.
Consequently, upon transposing $\widetilde{g}_{i,j}$ to $\widetilde{g}_{j,i}$,
\begin{equation} \label{eqn:QP}
\sum_{u \in \W_N} Q(x,u) P(u,y) = \prod_{j=1}^N \prod_{k=1}^m (1-a_kb_j) \sum_{u \in \W_N} \dt{f_{i,j}(u_j)} \dt{\widetilde{g}_{j,i}(-u_j)}.
\end{equation}

By changing variables $w \mapsto w/a_m$ in the integral defining $\widetilde{g}_{i,j}$ we find that
$$\widetilde{g}_{i,j}(u) = \frac{a_m^{i-j}}{2 \pi \mathbold{i}} \oint \limits_{|w|=R_2} dw \, w^u (wb_j)^{y_j}
\frac{\prod_{k=1}^j(w- 1/b_k)}{\prod_{k=1}^i (w- 1/b_k)} \, \frac{1}{w-a_m}, \quad R_2 > \max_k \{1/b_k \}.$$
We may remove $a_m^{i-j}$ from the determinant as it is a conjugation factor. So, if we define $g_{i,j}(u) = a_m^{j-i} \widetilde{g}_{i,j}(u)$, we have that
$$\pr{\vec{\Gb}(m)= y \vt \vec{\Gb}(0)= x} = \prod_{j=1}^N \prod_{k=1}^m (1-a_kb_j) \sum_{u \in \W_N} \dt{f_{i,j}(u_j)} \dt{g_{j,i}(-u_j)}.$$

\begin{lem} \label{lem:sbp1}
It holds that
$$\sum_{u \in \W_N} \dt{f_{i,j}(u_j)} \dt{g_{j,i}(-u_j)} = \sum_{u \in \W_N} \dt{f_{i,0}(u_j)} \dt{g_{0,i}(-u_j)}.$$
\end{lem}
We will prove this lemma later in $\S$\ref{sec:lemsbp1}.
For now, it allows us to complete the proof of Proposition \ref{prop:1} by using the Cauchy-Binet identity:
$$ \sum_{u \in \W_N} \dt{f_{i,0}(u_j)} \dt{g_{0,i}(-u_j)} = \dt{ \sum_{u \in \Z} f_{i,0}(u)g_{0,j}(-u)}.$$

Observe that
$$f_{i,0}(u) = \intz{R_1} \frac{dz}{z} \, \frac{z^u}{(zb_i)^{x_i}} \prod_{k=1}^i (z- 1/b_k)^{-1} a(z),$$
and $f_{i,0}(u)$ vanishes when $u < x_1$. Indeed, for $u < \min_i \{x_i\} = x_1$, the integrand decays
at least to the order $|z|^{-2}$ as $|z| \to \infty$ (recall $a(z)$ is bounded and analytic outside the unit disk).
So we may contract the contour to $\infty$ when $u < x_1$. Consequently,
\begin{align*}
&\sum_{u \in \Z} f_{i,0}(u) g_{0,i}(-u)  = \sum_{u \geq x_1} f_{i,0}(u) g_{0,i}(-u) \\
& = \frac{1}{(2 \pi \mathbold{i})^2} \oint \limits_{|z|=R_1} dz \oint \limits_{|w|=R_2} dw\, \frac{(b_jw)^{y_j}}{(b_iz)^{x_i}}
\frac{\prod_{k=1}^j (w-1/b_k)}{\prod_{k=1}^i (z-1/b_k)} \frac{a(z)}{z(w-a_m)} \times \left [ \sum_{u \geq x_1} (z/w)^u\right ].
\end{align*}
Arranging the contours such that $|z/w| = R_1 / R_2 < 1$, the above equals
$$\frac{1}{(2 \pi \mathbold{i})^2} \oint \limits_{|z|=R_1} dz \oint \limits_{|w|=R_2} dw\, \frac{(b_jw)^{y_j}}{(b_iz)^{x_i}}
\frac{\prod_{k=1}^j (w-1/b_k)}{\prod_{k=1}^i (z-1/b_k)} \frac{a(z)(z/w)^{x_1}}{z(w-z)(1-a_m/w)}\,.$$ 

The $z$-contour may be contracted to $\infty$ since the integrand decays at least to the order $|z|^{-2}$ (due to $x_1-x_i \leq 0$ for every $i$).
However, doing so encounters a pole at $z = w$ since $R_1 < R_2$. The residue there gives
$$\sum_{u \in \Z} f_{i,0}(u) g_{0,j}(-u) =
\intz{R} \frac{dw}{w} \frac{(b_jw)^{y_j}}{(b_iw)^{x_i}} \frac{\prod_{k=1}^j (w-1/b_k)}{\prod_{k=1}^i (w-1/b_k)} \frac{a(w)}{(1-a_m/w)}.$$
The quantity above is precisely $K(i,j)$, so from \eqref{eqn:QP} it follows that
$$\pr{\vec{\Gb}(m) = y \vt \vec{\Gb}(0) = x} = \prod_{j=1}^N \prod_{k=1}^m (1-a_kb_j) \dt{K(i,j)}.$$
Proposition \ref{prop:1} is thus proved.

\subsubsection{Proof of Lemma \ref{lem:sbp1}} \label{sec:lemsbp1}
Introduce the operators $\dr(b)$, for $b \in \C$, acting on functions $f: \Z \to \C$ according to
\begin{equation} \label{eqn:drb} \dr(b) f(x) = f(x+1) - \frac{1}{b} f(x).\end{equation}
We will call these operators derivatives. The operator $\dr(b)$ is invertible over the space of functions $f$
that decay rapidly at $-\infty$, namely those $f$ for which $|f(x)| \leq \rho^x$ as $x \to -\infty$ for some $\rho > 1/|b|$.
Then,
\begin{equation} \label{eqn:drinv} \dr(b)^{-1} f( x) = \sum_{n < 0} b^{n+1} f(x+n).\end{equation}
These operators commute over all values of $b$.

We observe that
$$ \dr(b_j) f_{i,j-1} = f_{i,j} \quad \text{and} \quad \dr(b_j) g_{j,i} = g_{j-1,i}.$$
Indeed, for $f_{i,j}$ we have that
$$ \dr(b_j) f_{i,j-1}(u) = \intz{R} dz\, \dr(b_j)[z^u](u) \frac{\prod_{k=1}^{j-1} (z- 1/b_k)}{\prod_{k=1}^i (z-1/b_k)} a(z).$$
Since $\dr(b_j)[z^u] = z^u (z - 1/b_j)$, the identity follows. The calculation involving $g_{j,i}$ is similar.

Next, we make use of the summation by parts identity
\begin{equation} \label{eqn:sbp}
\sum_{x=a}^b [\dr(c)f](x) g(-x) = \sum_{x=a}^b f(x) [\dr(c)g](-x) + f(b+1)g(-b) - f(a)g(-a+1).
\end{equation}

We can prove Lemma \ref{lem:sbp1} by repeatedly using the summation by parts identity to move derivatives from
the determinant involving $f_{i,j}$s to the one involving $g_{j,i}$s. The boundary terms need to be zero, which
will be the case if we move derivatives in the proper order.

The order is that, first, we remove the last derivatives from every $f_{i,j}$ by going down from column $N$ to $1$.
Then the $f_{i,j}$ will become $f_{i,j-1}$ and the $g_{j,i}$ will turn to $g_{j-1,i}$ along their respective columns.
After this, we move the last derivative again from column $N$ down to column 2, reducing $f_{i,j-1}$ to $f_{i,j-2}$
and $g_{j-1,i}$ to $g_{j-2,i}$. Continuing like this from columns $N$ down to $k$ for every $k = N, N-1, \ldots, 1$
gives the desired result.

The boundary terms, which are determinants, always vanish because some two consecutive columns are equal or a
column is identically zero (the latter occurs for columns $N$ and 1). This is best illustrated by the first 2 applications
of the summation by parts identity, as shown below.

\paragraph{\emph{First application of \eqref{eqn:sbp}.}}
\begin{align*}
& \sum_{u \in \W_N} \dt{f_{i,j}(u_j)} \dt{g_{j,i}(-u_j)} = \\
& \sum_{u \in W_{N-1}} \sum_{u_N = u_{N-1}}^{\infty} \dt{f_{i,1}(u_1) \cdots \dr(b_N)f_{i,N-1}(u_N)} \dt{g_{1,i}(-u_1) \cdots g_{N,i}(-u_N)} \\
& \overset{\eqref{eqn:sbp}}{=} \sum_{u \in \W_N} \dt{f_{i,1}(u_1) \cdots f_{i,N-1}(u_N)} \dt{g_{1,i}(-u_1) \cdots g_{i,N-1}(-u_N)} \\
& + \lim_{u_N \to \infty} \dt{f_{i,1}(u_1) \cdots f_{i,N-1}(u_N)} \dt{g_{1,i}(-u_1) \cdots g_{N,i}(-u_N)} \\
& - \dt{\cdots \underbrace{f_{i,N-1}(u_{N-1}) \, f_{i,N-1}(u_{N-1})}_{equal}} \dt{\cdots g_{N,i}(-u_{N-1}+1)}.
\end{align*}
Notice that $g_{N,i}(-u_N) = 0$ for $u_N \geq N+y_N$ because then the contour in its definition may be contracted to $\infty$;
so the first boundary term vanishes. The second boundary term vanishes due to the last two columns in its first determinant being equal.

\paragraph{\emph{Second application of \eqref{eqn:sbp}.}} Following the first application, our sum equals
$$\sum_{u \in \W_N} \dt{f_{i,1}(u_1)\cdots f_{i,N-1}(u_{N-1}) f_{i,N-1}(u_N)} \dt{g_{1,i}(-u_1) \cdots g_{N-1,i}(-u_{N-1}) g_{N-1,i}(-u_N)}.$$
Writing $f_{i,N-1}(u_{N-1}) = \dr(b_{N-1}) f_{i,N-2}(u_{N-1})$ and considering only the sum of the variable $u_{N-1}$ above, holding all others fixed,
and then applying summation by parts in the variable $u_{N-1}$, results in the $u_{N-1}$-sum
\begin{align*}
& \sum_{u_{N-1} = u_{N-2}}^{u_N} \dt{\cdots f_{i,N-2}(u_{N-2}) f_{i,N-2}(u_{N-1}) f_{i,N-1}(u_{N})} \times \\
& \qquad \qquad \dt{\cdots g_{N-2,i}(-u_{N-2}) g_{N-2,i}(-u_{N-1}) g_{N-1,i}(-u_N)}  \\
& + \dt{\cdots f_{i,N-2}(u_{N}+1) f_{i,N-1}(u_N)} \dt{\cdots \underbrace{g_{N-1,i}(-u_N) \, g_{N-1,i}(-u_N)}_{equal}}\\
& - \dt{\cdots \underbrace{f_{i,N-2}(u_{N-2}) \, f_{i,N-2}(u_{N-2})}_{equal}\,  f_{i,N-1}(u_N)} \dt{\cdots g_{N-1,i}(-u_{N-2}+1) g_{N-1,i}(-u_N)}.
\end{align*}
The boundary terms vanish as desired and the lemma is proved continuing in this way.

\subsection{Multi-spatial distribution at a single time} \label{sec:3}

\subsubsection{Proof of Theorem \ref{thm:2}}
The distribution function in Theorem \ref{thm:2} is obtained from the Markovian transition formula from Theorem \ref{thm:1}.
Using Theorem \ref{thm:1},
$$\pr{G(m,n_1) < h_1, \ldots, G(m,n_p) < h_p \vt \vec{\Gb}(0) =x} = \sum_{\substack{y \in \W_N \\ y_{n_k} < h_k }} \dt{M(i,j \vt x,y)}.$$
We may assume that $h_1 \leq h_2 \leq \cdots \leq h_p$.

The sum over $y \in \W_N$ can be performed from the last column (involving variable $y_N$) down to the first.
The summation over $y_N$ involves only the last column of $M$, which then moves into column $N$ by multi-linearity.
It equals
\begin{equation} \label{eqn:colNsum}
\sum_{y_{N-1} \leq y_N < h_p} \intz{R} \frac{dz}{z}\, \frac{(zb_N)^{y_N}}{(zb_i)^{x_i}}
\frac{\prod_{k=1}^N (z- 1/b_k)}{\prod_{k=1}^i (z-1/b_k)} \prod_{k=1}^m \frac{1-a_kb_j}{1-a_k/z}.
\end{equation}
Now
$$\sum_{y_{N-1} \leq y < h_p} (zb_N)^{y_N} = \frac{1}{b_N} \, \frac{(zb_N)^{h_p} - (zb_N)^{y_{N-1}}}{z- 1/b_N}.$$
As a result, \eqref{eqn:colNsum} becomes the difference of two terms: $(I) - (II)$ where
$$ (I) = \intz{R} dz\, \frac{(zb_N)^{h_p-1}}{(zb_i)^{x_i}} \frac{\prod_{k=1}^{N-1} (z- 1/b_k)}{\prod_{k=1}^i (z-1/b_k)} \prod_{k=1}^m \frac{1-a_kb_N}{1-a_k/z}$$
and
$$(II) = \intz{R} \frac{dz}{zb_N} \, \frac{(zb_N)^{y_{N-1}}}{(zb_i)^{x_i}} \frac{\prod_{k=1}^{N-1} (z- 1/b_k)}{\prod_{k=1}^i (z-1/b_k)} \prod_{k=1}^m \frac{1-a_kb_N}{1-a_k/z}.$$
Observe term $(II)$ is a multiple of column $N-1$ of $M$ by a factor of
$$\frac{b_N^{y_{N-1}-1}}{b_{N-1}^{y_{N-1}}} \prod_{k=1}^m \frac{1-a_kb_N}{1-a_k b_{N-1}}.$$
So it does not affect the determinant. Term $(I)$ is column $N$ of $F$. So,
$$\sum_{\substack{y \in \W_N \\ y_{n_k} < h_k}} \dt{M(i,j \vt x,y)} =
\sum_{\substack{y \in \W_{N-1} \\ y_{n_k} < h_k}} \dt{\underbrace{M(i,j \vt x,y)}_{j \leq N-1} \, F(i,N \vt x)}.$$

We can continue to perform sums in the above manner. The summation over $y_{N-1}$ will be over the range $y_{N-2} \leq y_{N-1} < h_p$.
Once we get to variable $y_{n_{p-1}}$, the range of summation becomes $y_{n_{p-1}-1} \leq y_{n_{p-1}} < h_{p-1}$.
The result after the summations over variables $y_N$ to $y_{N-\ell+1}$ are performed equals
$$ \sum_{\substack{y \in \W_{N-\ell} \\ y_{n_k} < h_k}} \dt{\underbrace{M(i,j \vt x,y)}_{j \leq N - \ell} \, \underbrace{F(i, j \vt x)}_{j > N-\ell}}.$$
Continuing in this way gives the determinantal expression in Theorem \ref{thm:2}.

\subsubsection{Fredholm determinant} \label{sec:fredholm}
The determinant of $F$ from Theorem \ref{thm:2} can be expressed as a Fredholm determinant in the following way.
Write $F = F_a + F_b$, where
$$ F_{\ell}(i,j \vt x) = \frac{1}{2 \pi \mathbold{i}} \oint \limits_{\gamma_{\ell}} dz\,
\frac{(zb_j)^{h(j)-1}}{(zb_i)^{x_i}} \frac{\prod_{k=1}^{j-1}(z-1/b_k)}{\prod_{k=1}^i (z- 1/b_k)} \prod_{k=1}^m \frac{1-a_kb_j}{1-a_k/z},$$
and $\gamma_b$ is a contour including only the poles at $z = 1/b_k$ (but none of the poles at $z = a_k$)
and $\gamma_a$ is the complementary contour
containing the poles at $z = a_k$. This is possible since $a_k < 1/b_j$ for every $k$ and $j$.

The matrix $F_b$ is lower triangular with 1s on the diagonal. Indeed, when $i < j$, the integrand of $F_b$ has no poles at $z = 1/b_k$
and the contour $\gamma_b$ may be contracted to a point. When $i=j$, there is a single pole of the integrand at $z = 1/b_j$
and, as the contour is contracted, the residue there gives
$$F_b(j,j) = \frac{(b_j^{-1}b_j)^{y_j}}{(b_j^{-1}b_j)^{x_j}} \prod_{k=1}^m \frac{1-a_kb_j}{1-a_kb_j} = 1.$$

Being an $N \times N$ lower triangular matrix with 1s on the diagonal, $F_b$ has determinant 1 and an inverse given by
$$F_b^{-1} = \sum_{k=0}^{N-1} (I-F_b)^k$$
since $(I-F_b)^N = 0$. As a result,
\begin{equation} \label{eqn:FredholmF}
\dt{F(i,j \vt x)} = \dt{I + F_b^{-1}F_a} = \dt{I + \sum_{k=0}^{N-1} (I-F_b)^k F_a}.
\end{equation}

The inverse of $F_b$ has a simple expression when the initial condition $x$ is zero.
In general, one can write a tractable expression when $x$ follows an arithmetic pattern such as $x_i = c +id$.
See \cite{MQR} for recent work on general initial conditions. The following expression is obtained
from the orthogonalization above and a conjugation of the resulting matrix.
\begin{prop} \label{prop:spatialFredholm}
Consider the polynuclear growth model with $\vec{\Gb}(0) = 0$. Then for $n_1 < n_2 < \cdots < n_p = N$,
$$\pr{G(m,n_1) \leq h_1, \ldots, G(m,n_p) \leq h_p} = \dt{I + F}_{N \times N},$$
where $F$ has a $p \times p$ block structure according to the partition $[N] = (0, n_1] \cup (n_1,n_2] \cup \cdots \cup (n_{p-1}, n_p]$
of the rows and columns. Let $F(r,i;s,j) = \ind{i \in (n_{r-1},n_r], \, j \in (n_{s-1},n_s]} F(i,j)$ denote the $(r,s)$ block of $F$. Then,
\begin{align*}
F(r,i;s,j) = & \, \ind{r > s} \, \frac{1}{2 \pi \mathbold{i}} \oint \limits_{\gamma_b} d \zeta \,
\prod_{k=j}^i (\zeta-1/b_k)^{-1} \zeta^{h_s- h_r} \;\; + \\
& \frac{1}{(2 \pi \mathbold{i})^2} \oint \limits_{\gamma_b} d \zeta\oint \limits_{\gamma_a} dz\,
\frac{\prod_{k=1}^{j-1} (z- 1/b_k)}{\prod_{k=1}^i (\zeta-1/b_k)} \frac{z^{h_s-1}}{\zeta^{h_r-1}}
\prod_{k=1}^m \frac{(1-a_k/\zeta)}{(1-a_k/z)} \frac{1}{z-\zeta}
\end{align*}
where $\gamma_a$ is a contour enclosing the poles only at $z = a_k$ and $\gamma_b$ is a contour enclosing poles only at $\zeta = 1/b_k$.
The inverse of the matrix $F_b$ above is
$$F_b^{-1}(i,j) = \frac{1}{2 \pi \mathbold{i}} \oint \limits_{\gamma_b} dz \, \frac{\prod_{k=1}^{j-1}(z-1/b_k)}{\prod_{k=1}^i(z-1/b_k)}
(z b_i)^{1-h(i)} \prod_{k=1}^m \frac{1-a_k/z}{1-a_kb_i}.$$
\end{prop}

We remark that asymptotic analysis of the matrix $F$ will lead to the kind of extended Airy kernels encountered
in limit distributions along space-like paths in the KPZ universality class; see for instance \cite{BFS, BoGo, BP, CoKPZ, IS, JoDPG, QuKPZ}.

\section{The two-time distribution} \label{partII}

\subsection{Two-time distribution of the inhomogeneous model}  \label{sec:4}
In this section we will prove Theorem \ref{thm:twotime}, building up to it along a sequence of lemmas.
We will then explain in $\S$\ref{sec:6} how the determinant from Theorem \ref{thm:twotime} can be expressed
as a Fredholm determinant, which is often better for doing asymptotics.

First, we introduce some notation to manage the upcoming calculations.
Define the following functions for $z \in \C$, $x \in \Z$ and positive integers $j,\ell$ and $m$.
\begin{align}
\label{eqn:aj} a_j(z \vt (\ell,m]) &= \prod_{k= \ell + 1}^m \frac{1-a_kb_j}{1-a_k/z} \quad \text{for}\;\; j\in [N] \;\text{and} \; \ell < m. \\
\label{eqn:wj} w_j(x \vt (\ell, m]) &= \intz{R} dz\, z^{x-1} \, a_j(z | (\ell, m]) \quad R > \max \{1/b_1, \ldots, 1/b_N\}.
\end{align}

Recall the commuting operators $\dr(b)$ from \eqref{eqn:drb} which act by $\dr(b)f(x) = f(x+1) - b^{-1}f(x)$.
In this section they will act on functions that vanish identically to the left of some integer, for
whom the inverse of $\dr(b)$ may be applied according to \eqref{eqn:drinv}. More specifically,
they will act on the functions $w_j(x)$ above, and note these functions vanish when $x < 0$
because the $z$-contour can then be contracted to infinity. Consequently, define the operators
\begin{equation} \label{eqn:drij}
\dr(b_{(i,j]}) = \prod_{k=1}^j \dr(b_k) \prod_{k=1}^i \dr(b_k)^{-1} \quad \text{for}\; i, j \in [N].
\end{equation}

\subsubsection{Three lemmas}

\begin{lem} \label{lem:A}
The transition matrix of $\vec{\Gb}$ from Theorem \ref{thm:1} can be expressed as
$$\pr{\vec{\Gb}(m) = y \vt \vec{\Gb}(\ell)=x} = \prod_{j=1}^n b_j^{y_j-x_j} \, \dt{\dr(b_{(i,j]}) \cdot w_j(y_j-x_i \vt (\ell,m])}.$$
\end{lem}

\begin{proof}
Observe from Theorem \ref{thm:1} that $\pr{\vec{\Gb}(m) = y \vt \vec{\Gb}(\ell) = x}$ equals
\begin{equation*}
\prod_{j=1}^N b_j^{y_j-x_j} \intz{R} dz z^{y_j-x_i-1} \frac{\prod_{k=1}^j (z-1/b_k)}{\prod_{k=1}^i (z-1/b_k)} a_j(z \vt (\ell,m]).
\end{equation*}
Next, observe that the function $x \mapsto z^x$ is an eigenfunction of $\dr(b)^{\pm 1}$ with eigenvalue $(z - 1/b)^{\pm 1}$, provided that $|z| > 1/b$.
This is clear for $\dr(b)$; for $\dr(b)^{-1}$ note that when $|z| > 1/b$,
$$ z^x(z-1/b)^{-1} = z^{x-1}(1 - 1/zb)^{-1} = \sum_{n < 0} b^{n+1} z^{n+x} = \dr(b)^{-1}[z^x](x).$$
From this fact we deduce that
$$z^{y_j-x_i} \, \frac{\prod_{k=1}^j (z-1/b_k)}{\prod_{k=1}^i (z-1/b_k)} = \dr(b_{(i,j]}) \big[z^x \big](x = y_j-x_i) \quad \text{for}\; |z| > \max \{1/b_1, \ldots, 1/b_N\}.$$
Finally, the operator $\dr(b)$ applied in the $x$-variable above commutes with the contour integration in the $z$-variable. Therefore,
\begin{align*}
& \intz{R} dz \, z^{y_j-x_i-1} \frac{\prod_{k=1}^j (z-1/b_k)}{\prod_{k=1}^i (z-1/b_k)} a_j(z \vt (\ell,m]) = \\
& \dr(b_{(i,j]}) \left [ \intz{R} dz\, z^{x-1} a_j(z \vt (\ell,m]) \right](x = y_j-x_i) = \\
& \dr(b_{(i,j]}) \cdot w_j(y_j-x_i \vt (\ell,m]). \qedhere
\end{align*}
\end{proof}

\begin{lem} \label{lem:B}
The following identity holds for $x,z \in \W_N$ and $1 \leq n < N$.
\begin{align*}
&\sum_{\substack {y \in \W_N \\ y_n < h}} \dt{\dr(b_{(i,j]}) \cdot w_j(y_j-x_i \vt (0,m])}_{i,j} \dt{\dr(b_{(i,j]}) \cdot w_j(z_j-y_i \vt (m,M])}_{i,j}  = \\
&\sum_{\substack {y \in \W_N \\ y_n < h}}\dt{\dr(b_{(i,n]}) \cdot w_j(y_j-x_i \vt (0,m])}_{i,j} \dt{\dr(b_{(n,j]}) \cdot w_j(z_j-y_i \vt (m,M])}_{i,j}.
\end{align*}
\end{lem}

\begin{proof}
We may write $w_j(x | (0,m]) = \prod_{k=1}^m (1-a_kb_j) f(x)$ for a function $f :\Z \to \C$ that does not depend on $j$, and likewise
$w_j(x | (m,M]) = \prod_{k=m+1}^M (1-a_kb_j) g(x)$. The functions $f$ and $g$ vanish on the negative integers. The factors of $\prod_{k} (1-a_kb_j)$
can be pulled out of the determinants, and they cancel from both sides of the identity. Moreover, upon conditioning on the value of $y_n$ and transposing the 2nd determinant, it
is then enough to show that
\begin{align*}
&\sum_{\substack {y \in \W_N \\ y_n = h}} \dt{\dr(b_{(i,j]}) f(y_j-x_i)} \dt{\dr(b_{(j,i]}) g(z_i-y_j)} = \\
&\sum_{\substack {y \in \W_N \\ y_n = h}}\dt{\dr(b_{(i,n]}) f(y_j-x_i)} \dt{\dr(b_{(n,i]}) g(z_i-y_j)}.
\end{align*}

Now we can use the summation by parts identity \eqref{eqn:sbp} to move derivatives around.
For column $j > n$, we would like to move its last $j-n$ derivatives ($\dr(b_j), \ldots, \dr(b_{n+1})$)
from the 1st determinant to the 2nd. For column $j < n$, we would like to move its last $n-j$
derivatives from the 2nd determinant to the 1st. In doing so we have to ensure that all boundary
terms from the summation by parts identity are zero. This will be the case so long as the derivatives
are moved in the proper order.

The proper order is to first move the final derivatives, $\dr(b_N), \ldots, \dr(b_{n+1})$, from columns
$N, \ldots, n+1$ of the 1st determinant to the 2nd, whereupon the total derivative along those columns
in the 1st determinant becomes $\dr(b_{(i,j-1]})$ and in the 2nd determinant it becomes $\dr(b_{(j-1,i]})$.
Next, continue to move final derivatives from columns $N$ to $n+1$ of the 1st determinant
to the 2nd, and then again from columns $N$ to $n+2$, and so on for a total of $N-n$ rounds.
After these rounds the total derivative along column $j > n$ of the 1st determinant becomes $\dr(b_{(i,n]})$
and in the 2nd determinant it becomes $\dr(b_{(n,i]})$, as desired.

In order to move the derivatives along the first $n-1$ columns, write $\dr(b_{(j,i]}) g = \dr(b_{(j,n]}) \cdot \dr(b_{(n,i]}) g$.
The derivatives $\dr(b_{(j,n]})$ will be moved from the 2nd determinant to the 1st. Note that these are
indeed derivatives since $j < n$.
First move the leading derivatives, $\dr(b_2), \ldots, \dr(b_n)$, from columns $1, \ldots, n-1$
of the 2nd determinant to the 1st. Then move the new leading derivatives, which are $\dr(b_3), \ldots, \dr(b_n)$,
along columns $1$ through to $n-2$ of the 1st determinant to the 2nd, and continue like this for $n-1$ rounds
to get the desired form.

The boundary terms will be zero during each application of the summation by parts identity.
The reasoning is like in the proof of Lemma \ref{lem:sbp1}. When operating on column $j$
for $1 < j < N$, a boundary term of the form $\dt{\cdot} \dt{\cdot} - \dt{\cdot}\dt{\cdot}$
will be zero because two consecutive columns will be equal in one of the determinants from each pair (recall
the second application of \eqref{eqn:sbp} in the proof of Lemma \ref{lem:sbp1}).
For columns $j = 1$ or $N$, the boundary term will be zero because in one of the terms, $\dt{\cdot} \dt{\cdot}$,
two consecutive columns will be equal while for the other the $j$-th column itself will be zero due to $f$ and $g$
vanishing identically on the negative integers (recall the first application of \eqref{eqn:sbp} in the proof of Lemma \ref{lem:sbp1}).

The lemma is proved by carrying out this routine.
\end{proof}

\begin{lem} \label{lem:C}
	The following identity also holds for $y \in \W_N$.
	$$\sum_{\substack{z \in \W_N \\ z_N < H}} \prod_{j=1}^n b_j^{z_j} \dt{\dr(b_{(n,j]} w_j(z_j-y_i \vt a_{(\ell,m]})} =
		\prod_{j=1}^{n}b_j^{H-1} \dt{\dr(b_{(n,j]} w_j(H-y_i \vt a_{(\ell,m]})}.$$
\end{lem}

\begin{proof}
We perform the sums beginning with variable $z_N$ down to $z_1$. The summation over $z_N$ gives
$$ \sum_{z_{N-1} \leq z_N < H} b_N^{z_N} \, \dt{\underbrace{\dr(b_{(n,j)}) w_j(z_j-y_i \vt (\ell,m])}_{j < N}\, \dr(b_{(n,N)}w_N(z_N-y_i \vt (\ell, m]))}.$$
The sum moves into column $N$ with a factor of $b_N^{z_n}$ in front. Write $\dr(b_{(n,N]})w_N(z_N-y_i) = \dr(b_N) f(z_N)$
with $f(z) = \dr(b_{(n,N-1]}) w_N(z-y_i \vt (\ell,m])$. Then the sum to evaluate is
$$\sum_{z_{N-1} \leq z < H} b_N^z \dr(b_N)f(z) = b_N^{H-1}f(H) - b_N^{z_{N-1}-1}f(z_{N-1}).$$
This follows from the general identity that $\sum_{u \leq z < v} b^z \dr(b)g(z) = b^{v-1} g(v) -b^{u-1} g(u)$.

Now observe that $b_N^{z_{N-1}-1}f(z_{N-1})$ is a scalar multiple of the $(N-1)$-th column of the determinant.
Indeed, $w_N(z) = \prod_{\ell < k \leq m} \frac{1-a_kb_N}{1-a_kb_{N-1}}\, w_{N-1}(z)$, and so
$$b_N^{z_{N-1}-1}f(z_{N-1}) = \lambda \, \dr(b_{(n,N-1]})w_{N-1}(z_{N-1}-y_i)$$
for $\lambda = b_N^{z_{N-1}-1}\prod_k \frac{1-a_kb_N}{1-a_kb_{N-1}}$.
So this term does not affect the determinant and, following the sum over $z_N$, we find that
\begin{align*}
& \sum_{\substack{z \in \W_N \\ z_N < H}} \prod_{j=1}^n b_j^{z_j} \dt{\dr(b_{(n,j]} w_j(z_j-y_i \vt a_{(\ell,m]})} = \\
&b_N^{H-1} \sum_{\substack{z \in \W_{N-1}\\ z_{N-1 < H}}} \prod_j b_j^{z_j} \, \dt{\underbrace{\dr(b_{(n,j)}) w_j(z_j-y_i \vt (\ell,m])}_{j < N}\, \dr(b_{(n,N-1)}w_N(H-y_i \vt (\ell, m]))}.
\end{align*}
The lemma follows by iterating the above procedure for the remaining columns.
\end{proof}

\subsubsection{Proof of Theorem \ref{thm:twotime}}
By Theorem \ref{thm:1}, the two-time distribution function $$\pr{\Gb(m,n) < h, \Gb(M,N) < H \vt \vec{\Gb}(0)=x}$$ equals
\begin{align*}
& \sum_{\substack{y \in \W_N \\ y_n < h}} \sum_{\substack{z \in \W_N \\ z_N < H}} \pr{\vec{\Gb}(m)=y \vt \vec{\Gb}(0)=x} \pr{\vec{\Gb}(M)=z \vt \vec{\Gb}(m)=y} \\
& = \sum_{\substack{y \in \W_N \\ y_n < h}} \sum_{\substack{z \in \W_N \\ z_N < H}}
\prod_{j=1}^N b_j^{z_j-x_j}\, \dt{\dr(b_{(i,j]}) \cdot w_j(y_j-x_i \vt (0,m])} \dt{\dr(b_{(i,j]}) \cdot w_j(z_j-y_i \vt (m,M])}
\end{align*}
We have used Lemma \ref{lem:A} to rewrite the transition matrix of $\vec{\Gb}$ in terms of the operators $\dr(b)$ and functions $w_j$.
Lemma \ref{lem:B} combined with Lemma \ref{lem:C} then leads to the following expression for this sum:
\begin{equation} \label{eqn:twotime}
\prod_{j=1}^N b_j^{H-1-x_j} \sum_{\substack{y \in \W_N \\ y_n < h}} \dt{\dr(b_{(i,n]}) \cdot w_j(y_j-x_i \vt (0,m])} \dt{\dr(b_{(n,j-1]}) \cdot w_j(H-y_i \vt (m,M])}.
\end{equation}

We can express the sum over $y \in \W_N$ with the constraint that $y_n < h$ as a contour integral
involving the unconstrained sum. It works as follows. For $y \in \W_N$, the constraint that
$y_n < h$ is the same as $\# \{j: y_j < h\} \geq n$. We may then write, for $s > 1$,
\begin{align*}
\ind{\# \{j: y_j < h\} \geq n} & = \frac{1}{2 \pi \mathbold{i}} \oint \limits_{|\theta|=s} \frac{\theta^{\# \{j: y_j < h\}-n}}{\theta-1} \\
& = \frac{1}{2 \pi \mathbold{i}} \oint \limits_{|\theta|=s} \frac{\prod_{j=1}^N \theta^{\ind{y_j < h}}}{\theta^n(\theta -1)}.
\end{align*}
The identity can be seen by expanding $(\theta-1)^{-1}$ is powers of $\theta^{-1}$ and then
interchanging summation with integration. It now follows that
\begin{align*}
&\eqref{eqn:twotime} = \frac{1}{2 \pi \mathbold{i}} \oint \limits_{|\theta|=s} \frac{S}{\theta^n(\theta -1)}, \; \text{where}\\
& S = \sum_{y \in \W_N} \prod_{j=1}^N b_j^{H-1-x_j} \theta^{\ind{y_j<h}}
\dt{\dr(b_{(i,n]}) w_j(y_j-x_i | (0,m])} \dt{\dr(b_{(n,j-1]}) w_j(H-y_i | (m,M])}.
\end{align*}
The proof of Theorem \ref{thm:twotime} is complete once we show that $S$ equals $\dt{\theta L_1 - L_2}$.
Indeed, we can insert $\theta^{-n}$ into the determinant by plugging a factor of $\theta^{-1}$ into the
first $n$ rows, whereupon we find that
$$\theta^{-n}S = \dt{\theta^{-\ind{i \leq n}} (\theta L_1 - L_2)} = \dt{\theta^{\ind{i>n}}L_1 - \theta^{-\ind{i\leq n}}L_2}.$$
So, we are finished after proving the following lemma.

\begin{lem}
The sum $S$ equals $\dt{\theta L_1 - L_2}$ where $L_1$ and $L_2$ are as given by Theorem \ref{thm:twotime}.
\end{lem}

\begin{proof}
The sum can be evaluated with the Cauchy-Binet identity. First, write $w_j(x |\, (\ell,m]) = \prod_{k=\ell+1}^m (1-a_kb_j) w(x |\, (\ell,m])$
by pulling out factors of $(1-a_kb_j)$ from $a_j(z| (\ell,m])$ so that $w(x)$ no longer depends on $j$. Then,
$$S = Z \sum_{y \in \W_N} \theta^{\ind{y_j<h}} \dt{\dr(b_{(i,n]}) w(y_j-x_i | (0,m])} \dt{\dr(b_{(n,j-1]}) w(H-y_i | (m,M])}$$
where $Z = \prod_{j=1}^N b_j^{H-1-x_j} \prod_{k=1}^M(1-a_kb_j)$. We will put $Z$ back into the determinant later.

The sum, apart from $Z$, now looks like $\sum_{y \in \W_N} \dt{f(i, y_j)} \dt{g(y_i,j)}$ where
\begin{align*}
& f(i,y) = \dr(b_{(i,n]}) w(y-x_i \vt (0,m]) \, \theta^{\ind{y<h}},\quad g(y,j)  = \dr(b_{(n,j-1]}) w(H-y \vt (m,M]).
\end{align*}
The Cauchy-Binet identity implies that
\begin{align} \nonumber
& \sum_{y \in \W_N} \dt{f(i, y_j)} \dt{g(y_i,j)}  = \dt{\sum_{y \in \Z} f(i,y) g(y,j)} \\
\label{eqn:CB} & = \; \dt{\sum_{y \in \Z} \theta^{\ind{y<0}} \, \dr(b_{(i,n]}) w(y+h-x_i \vt (0,m]) \dr(b_{(n,j-1]}) w(H-h-y \vt (m,M])}.
\end{align}

Looking at \eqref{eqn:CB}, we see that the matrix inside the determinant has the form $\theta \widetilde{L}_1(i,j) + \widetilde{L}_2(i,j)$
where $\widetilde{L}_1$ is obtained from the summation over $y < 0$ and $\widetilde{L}_2$ over $y \geq 0$. We can
insert $Z$ back into the determinant by breaking up each term
$$ b_{\ell}^{H-1-x_{\ell}} \prod_{k=1}^M (1 - a_kb_{\ell}) =
b^{h-1-x_{\ell}} \prod_{k \leq m} (1-a_kb_{\ell}) \, \times \, b_{\ell}^{H-h} \prod_{k > m}(1- a_k b_{\ell}),$$
and putting $b_j^{H-h} \prod_{k>m} (1-a_k b_j)$ into column $j$ while $b_i^{h-1-x_i} \prod_{k \leq m}(1-a_kb_i)$ into
row $i$ of $\widehat{L}_1$ and $\widehat{L}_2$. Comparing the resulting matrices with $L_1$ and $L_2$,
it suffices to show the following to conclude the proof.
\begin{align} \label{eqn:tildeL}
\widehat{L}_1(i,j) & = \frac{1}{(2\pi \mathbold{i})^2} \oint \limits_{|z|=R_1} dz \oint \limits_{|w|=R_2} dw\, \\ \nonumber
& \frac{z^{h-1-x_i}w^{H-h}}{z-w} \frac{\prod_{k=1}^n (z-1/b_k)}{\prod_{k=1}^i(z-1/b_k)} \frac{\prod_{k=1}^{j-1} (w-1/b_k)}{\prod_{k=1}^n(w-1/b_k)} a(z \vt (0,m]) a(w \vt (m,M]),
\end{align}
where $R_1 > R_2 > \max_k \{1/b_k\}$ and $a(z \vt (\ell,m]) = \prod_{k=\ell+1}^m(1-a_k/z)^{-1}$.
Matrix $\widehat{L}_2$ is the same except it has a minus sign in front and the contours are arranged such that $R_2 > R_1$.

This representation of $\widetilde{L}_k$ is proven by using Lemma \ref{lem:A} to express $\dr(b_{(i,j]}) w(x \vt (\ell,m])$
as a contour integral:
\begin{align*}
\dr(b_{(i,n]}) w(y+h-x_i \vt (0,m]) &= \intz{R_1} dz\, z^{y+h-1-x_i} \frac{\prod_{k=1}^n (z-1/b_k)}{\prod_{k=1}^i(z-1/b_k)} a(z \vt (0,m]) \\
\dr(b_{(n,j-1]}) w(H-h-y \vt (m,M]) &= \frac{1}{2\pi \mathbold{i}} \oint \limits_{|w|=R_2} dw\, w^{H-h-1-y} \frac{\prod_{k=1}^{j-1} (w-1/b_k)}{\prod_{k=1}^n(w-1/b_k)} a(w \vt (m,M]).
\end{align*}
Their product then equals
\begin{align*}
&  \frac{1}{(2\pi \mathbold{i})^2} \oint \limits_{|z|=R_1} dz \oint \limits_{|w|=R_2} dw\, (z/w)^y \, \times \\
 & z^{h-1-x_i}\, w^{H-h-1} \, \frac{\prod_{k=1}^n (z-1/b_k)}{\prod_{k=1}^i(z-1/b_k)} \frac{\prod_{k=1}^{j-1} (w-1/b_k)}{\prod_{k=1}^n(w-1/b_k)}  \, a(z \vt (0,m]) \, a(w \vt (m,M]).
 \end{align*}
 The summation over $y < 0$ and $y \geq 0$ may be interchanged with integration to give
\begin{description}
 	\item[For $\widetilde{L}_1$: ] $\sum_{y < 0} (z/w)^y = \frac{w}{z-w}$ provided that $|w/z| = R_2/R_1 < 1$;
 	\item[For $\widetilde{L}_2$:] $\sum_{y \geq 0} (z/w)^y = - \frac{w}{z-w}$ provided that $|z/w| = R_1/R_2 < 1$.
\end{description}
Performing these sums results in the representation \eqref{eqn:tildeL} for $\widetilde{L}_k$.
\end{proof}

\subsubsection{Orthogonalization and a Fredholm determinant} \label{sec:6}
We will explain a general method to express the determinant in the formula for the two-time distribution of $\Gb$ 
in terms of a Fredholm determinant. It takes a simple form when $\vec{\Gb}(0) = 0$, which is presented in
Theorem \ref{thm:4}. It will be used for doing asymptotics in the following sections.

\paragraph{\textbf{A general Fredholm determinant form.}}
Consider the determinant $\dt{\theta^{-\ind{i \leq n}}(\theta L_1 - L_2)}$ from Theorem \ref{thm:twotime}.
We will see that the matrix $L_1$ can be expressed as
$$L_1(i,j) = \ind{i \leq n, \,  j \leq n} F_b(i,j) + (\text{4 other matrices})$$
where $F_b$ is a lower triangular matrix with 1s on the diagonal (in fact, the same $F_b$ from $\S$\ref{sec:fredholm}).
The matrix $L_2$ can similarly be decomposed as
$$- L_2(i,j) = \ind{i > n,\, j> n} F_b(i,j) + (\text{4 other matrices}).$$
This leads to the not-so-obvious fact that $\theta^{-\ind{i \leq n}}(\theta L_1 - L_2)$ equals
$$ (\ind{i, j \leq n} + \ind{i,j > n})F_b + (\text{many other matrices}),$$
and the matrix $T = (\ind{i, j \leq n} + \ind{i,j > n})F_b$ is lower triangular with 1s on its diagonal.
It can then be taken out of the determinant to get that
$$\dt{\theta^{-\ind{i \leq n}}(\theta L_1 - L_2)} = \dt{I + F(\theta)}$$
for some matrix $F(\theta)$.

The matrix $T$ in fact has a $ 2 \times 2$ block triangular form
$$ T = \begin{bmatrix}
T_u & 0 \\
0 & T_{\ell}
\end{bmatrix}$$
where $T_u$ and $T_{\ell}$ are the upper and lower blocks of $F_b$ according to
the partition $[N] = (0,n] \cup (n,N]$ of the row and columns. The inverting out of $T$ can then
be done by left multiplication by the matrix
$$A = \begin{bmatrix}
	T_u^{-1} & 0 \\
	0 & I
\end{bmatrix}$$
and right multiplication by
$$ B = \begin{bmatrix}
I & 0 \\
0 & T_{\ell}^{-1}
\end{bmatrix},$$
both of which have determinant 1. In this way one finds that
$$\dt{\theta^{-\ind{i \leq n}}(\theta L_1 - L_2)} = \dt{\theta^{-\ind{i \leq n}}(\theta (AL_1B) - (AL_2B))},$$
with $AL_1B = \ind{i=j, \, i\leq n} + F_1$ and $-AL_2B = \ind{i=j, \, i > n} + F_2$. Then, crucially,
$$\theta^{-\ind{i \leq n}}(\theta \ind{i=j,\, i \leq n} + \ind{i=j,\, i>n}) = I$$
and $F(\theta) = \theta^{\ind{i>n}}F_1 + \theta^{-\ind{i \leq n}} F_2$.

To express $L_1$ as above one proceeds in the following manner; the procedure for $L_2$
is much the same. Decompose the contour $\{ |z| = R_1\} = \gamma_a \cup \gamma_b$
where $\gamma_b$ contains only the poles at $z = 1/b_k$ and $\gamma_a$ contains
the complementary poles at $z = a_k$. Similarly, break up $\{|w| = R_2\} = \gamma'_a \cup \gamma'_b$. 
The condition $R_1 > R_2$ means that $\gamma_k$ contains $\gamma'_k$ on the inside.
Due to this decomposition of contours,
$$L_1 = F_{b,b} + F_{b,a} + F_{a,b} + F_{a,a}$$
where
\begin{align*}
F_{b,b}(i,j) &= \frac{1}{(2 \pi \mathbold{i})^2} \oint \limits_{\gamma_b} dz \oint \limits_{\gamma'_b} dw \,
(b_iz)^{h-1} (b_jw)^{H-h} (b_i z)^{-x_i}\, a_i(z \vt (0,m]) \, a_j(w \vt (m,M]) \times \\
& \frac{1}{z-w} \, \frac{\prod_{k=1}^n(z-1/b_k)}{\prod_{k=1}^i(z-1/b_k)} \frac{\prod_{k=1}^{j-1}(w-1/b_k)}{\prod_{k=1}^n(w-1/b_k)},
\end{align*}
and the rest are similar with respect to the remaining contours.

The $\gamma'_b$ contour can be contracted to a point when $j > n$ because then there are no $w$-poles. So,
$$F_{b,b}(i,j) = \ind{j \leq n} F_{b,b}(i,j) = \ind{i \leq n,\, j \leq n} F_{b,b}(i,j) + \ind{i > n,\, j \leq n} F_{b,b}(i,j).$$
When $i \leq n$, the $\gamma_b$ contour can also be contracted to a point but doing so incurs a residue
at $z=w$ due to the ordering of the contours. The residue there is precisely $F_b(i,j)$, so that
$$F_{b,b}(i,j) = \ind{i,j \leq n} F_b(i,j) + \ind{i > n,\, j \leq n} F_{b,b}(i,j).$$
This is the decomposition we wanted, which then leads to a Fredholm determinant.

\paragraph{\textbf{Orthogonalization when the initial condition is zero.}}
Consider the two-time distribution when $\vec{\Gb}(0) = 0$, for which we perform the orthogonalization explicitly.
We first observe a symmetry relating
$L_1$ and $L_2$ that allows to consider the orthogonalization only for $L_1$.
Write $L_1(i,j)$ in the following suggestive form:
\begin{align*}
L_1(i,j) &= \frac{1}{(2 \pi \mathbold{i})^2} \oint \limits_{|z|=R_1} dz \oint \limits_{|w|=R_2} dw\,
\frac{(b_iz)^{h-1} (b_jw)^{H-h}}{(z-w)} \frac{\prod_{k \in [n]}(z-1/b_k)}{\prod_{k \in [i]}(z-1/b_k)} \frac{\prod_{k \in [j-1]}(w-1/b_k)}{\prod_{k \in [n]}(w-1/b_k)} \\
& \times \prod_{k \in [m]} \frac{1-a_kb_i}{1-a_k/z} \prod_{k \in [M]\setminus [m]} \frac{1-a_kb_j}{1-a_k/w}.
\end{align*}
Here $[n]$ denotes the set $\{1, 2, \ldots, n\}$ of integers.

In the above, $L_1$ depends on the parameters $i,j, n, m, h, N, M, H$, and apart from $h$ and $H$,
the dependence of the other parameters is through the subsets $[i], [j-1], [n], [m], [N]$ and $[M]$.
Thus, $L_1(i,j) = L_1([i], [j-1], [n], [N], [m], [M], h-1, H)$. Now $L_2$ can be expressed in terms of
$L_1$ but with a different set of parameters, namely
\begin{equation} \label{eqn:L2L1}
-L_2(i,j) = L_1([N] \setminus [j-1], [N] \setminus [i], [N]\setminus [n], [N], [M] \setminus [m], [M], H-h, H).
\end{equation}
This follows from exchanging the contour variables $z \leftrightarrow w$ in the integral for $L_2$
and then substituting the complementary parameters.

What this means is that if we find matrices $A$ and $B$ so that $AL_1B(i,j) = \ind{i=j, \, i \leq n} + F_1(i,j)$,
then it will automatically be the case that $- AL_2 B(i,j) =  \ind{i=j, \, i > n} + F_2$ with
$F_2$ related to $F_1$ according to \eqref{eqn:L2L1}. The indicator $\ind{i=j, \, i \leq n}$ should be
read as $\ind{\# [i] = \# [j-1] + 1, \, \# [i] \leq \# [n]}$, so that after substituting $[N] \setminus [j-1]$
for $[i]$, $[N] \setminus [i]$ for $[j-1]$ and $[N] \setminus [n]$ for $[n]$, it turns to $\ind{i=j, \, j > n}$
as needed.

Define the orthogonalizing matrices $A$ and $B$ by
\begin{align} \label{eqn:matA}
A(i,j) &= \frac{1}{2 \pi \mathbold{i}} \oint_{\gamma_b} d \zeta \,
\frac{\prod_{k \in [j-1]} (\zeta - 1/b_k)}{\prod_{k \in [i]} (\zeta - 1/b_k)} (b_j \zeta)^{1-h} \prod_{k \in [m]} \frac{1-a_k/\zeta}{1-a_k b_j}\\
\label{eqn:matB}
B(i,j) &= \frac{1}{2 \pi \mathbold{i}} \oint_{\gamma_b} d \omega \,
\frac{\prod_{k \in [j-1]} (\omega - 1/b_k)}{\prod_{k \in [i]} (\omega - 1/b_k)} (b_i \omega)^{h-H}
\prod_{k \in [M]\setminus [m]} \frac{1-a_k/\omega}{1-a_k b_i}.
\end{align}
The contour $\gamma_b$ encloses all the poles at $1/b_k$.

Observe $A(i,i) = B(i,i) = 1$ by a residue calculation with a simple pole at $1/b_i$.
Observe also that $A(i,j) = B(i,j) = 0$ when $j > i$ because the contours lack poles and may be contracted to a point.
So $A$ and $B$ are lower triangular with ones on the diagonal, as required for orthogonalization.

\begin{thm} \label{thm:4}
The two-time distribution function when $\vec{\Gb}(0)=0$ is given by
$$\pr{G(m,n) \leq h, G(M,N) \leq H} = \frac{1}{2 \pi \mathbold{i}} \oint \limits_{|\theta| = s}
\frac{\dt{I + \theta^{\ind{i > n}} F_1 + \theta^{-\ind{i \leq n}} F_2}}{\theta - 1}.$$
The matrix $F_2$ is related to $F_1$ by \eqref{eqn:L2L1}.
These matrices are sums, $F_1 = J_1 - J_2 + J_3$ and $F_2 = J_2 - J_1 - J_4$, with the $J$s given by the following formulas.

For $z \in \C$, $h \in \Z$ and subsets $S \subset [N]$ and $T \subset [M]$, define
\begin{equation} \label{eqn:GST}
	G(z \vt S, T, h) = z^{h} \prod_{k \in S} (z - 1/b_k) \prod_{k \in T} (1 - a_k/z)^{-1}.
\end{equation}
\begin{align*}
J_1(i,j) &= \ind{j \leq n} \, \frac{1}{(2 \pi \mathbold{i})^2}\oint_{\gamma_b} d \zeta \, \oint_{\Gamma_a} dz \
\frac{G (z \vt [j-1], [m], h-1)}{G (\zeta \vt [i], [m], h-1 )\, (z-\zeta)}\\
J_2(i,j) & = \ind{i > n} \, \frac{1}{(2 \pi \mathbold{i})^2} \oint_{\gamma_b} d\omega \, \oint_{\Gamma_a} dw \,
\frac{G(w \vt [N] \setminus [i], [M] \setminus [m], H-h)}{G(\omega \vt [N]\setminus [j-1], [M] \setminus [m], H-h) \, (w-\omega)}\\
J_3(i,j) & = \frac{1}{(2 \pi \mathbold{i})^4} \oint_{\gamma_b} d\zeta \, \oint_{\gamma_b} d \omega \, \oint_{\Gamma_a} dz\,  \oint_{\Gamma'_a} dw  \\
& \frac{G(z \vt [n],[m], h-1) G(w \vt [N]\setminus [n], [M] \setminus [m], H-h)}
{G(\zeta \vt [i],[m], h) \, G(\omega \vt [N]\setminus [j-1], [M] \setminus [m], H-h)\,(z-\zeta)(w-\omega)(z-w)}
\end{align*}
The contour $\gamma_b$ encloses only the poles at every $1/b_k$. The contours $\Gamma_a$ and $\Gamma'_a$ enclose only
the poles at every $a_k$. In $J_3$, $\Gamma_a$ contains $\Gamma'_a$ (so $|z| > |w|$).

The matrix $J_4$ looks the same as $J_3$ except the $z$ and $w$ contours are reversed
so that $\Gamma'_a$ contains $\Gamma_a$ (so $|w| > |z|$).
\end{thm}

\begin{proof}
Following the discussion above, it is enough to show that $AL_1B = \ind{i=j,\, i \leq n} + F_1$.
Multiplying $L_1$ by $A$ and $B$ and simplifying gives the following:
\begin{align*}
	AL_1B(i,j) &= \frac{1}{(2 \pi \mathbold{i})^4} \oint_{\gamma_b} d\zeta \, \oint_{\gamma_b} d \omega \, \oint_{|z|=R_1} dz \, \oint_{|w|=R_2}dw\, 
	\frac{(z/\zeta)^{h-1} (w/\omega)^{H-h}}{(z-w)} \\
	\times & \frac{\prod_{k \in [n]}(z - 1/b_k) \prod_{k \in [j-1]}(\omega - 1/b_k)}{\prod_{k \in [n]}(w-1/b_k) \prod_{k \in [i]}(\zeta - 1/b_k)}
	\prod_{k \in [m]}\frac{1-a_k/\zeta}{1-a_k/z} \prod_{k \in [M]\setminus [m]} \frac{1-a_k/\omega}{1-a_k/w}\\
	\times & \left [ \sum_{r,s=1}^{N} \frac{\prod_{k=1}^{r-1}(\zeta - 1/b_k)}{\prod_{k=1}^r (z-1/b_k)} \cdot
	\frac{\prod_{k=1}^{s-1}(w - 1/b_k)}{\prod_{k=1}^r (\omega-1/b_k)} \right ].
\end{align*}

In order to evaluate the double sum, observe that
$$ \frac{\prod_{k=1}^{\ell-1}(x- c_k)}{\prod_{k=1}^{\ell}(y - c_k)} =
\frac{1}{y-x} \left [ \prod_{k=1}^{\ell-1} \frac{x-c_k}{y-c_k} - \prod_{k=1}^{\ell} \frac{x-c_k}{y-c_k}\right ].$$
So the double sum telescopes to
$$\frac{1}{(z-\zeta)(w-\omega)}
\left ( 1 - \prod_{k=1}^N \frac{\zeta - 1/b_k}{z - 1/b_k}\right ) \left ( \prod_{k=1}^N \frac{w-1/b_k}{\omega - 1/b_k}-1\right).$$

Multiply the product above into 4 terms and plug them into the integral defining $AL_1B(i,j)$ above.
This turns the integral into 4 integrals, and the only non-zero one is the integral with the term
$1 \times \prod_{k=1}^N \frac{w-1/b_k}{\omega - 1/b_k}$. The other integrals are zero because
either the $\zeta$-contour or the $\omega$-contour can be contracted due to not having poles.

Recalling the $G$-function \eqref{eqn:GST}, $AL_1B(i,j)$ then equals
\begin{align} \label{eqn:ALB}
	AL_1B(i,j) &= \frac{1}{(2 \pi \mathbold{i})^4} \oint_{\gamma_b} d\zeta \, \oint_{\gamma_b} d \omega \, \oint_{|z|=R_1} dz \, \oint_{|w|=R_2}dw
	\\ \nonumber
	& \frac{G(z \vt[n], [m], h-1 ) G(w \vt [N] \setminus [n], [M] \setminus [m], H-h)}
	{G(\zeta \vt[i], [m], h-1 ) G(\omega \vt [N] \setminus [j-1], [M] \setminus [m], H-h)\, (z-w)(z-\zeta)(w-\omega)}.
\end{align}

Break up the contour $\{|z|=R_1\}$ into $\Gamma_a \cup \Gamma_b$ where $\Gamma_a$ encloses on the poles at $z = a_k$
and $\Gamma_b$ only the poles at $z = 1/b_k$. The contour $\Gamma_b$ should also contain the $\zeta$-contour $\gamma_b$.
Do the same for $\{|w|=R_2\}$ into $\Gamma'_a \cup \Gamma'_b$.
The condition $R_1 > R_2$ stipulates $\Gamma_a$ contains $\Gamma'_a$ and $\Gamma_b$ contains $\Gamma'_b$.
With this decomposition,
$$AL_1B = J_{b,b} + J_{a,b} + J_{b,a} + J_{a,a}$$
where $J_{x,y}(i,j)$ is the integral \eqref{eqn:ALB} but with the $z$-integral over $\Gamma_x$ and the $w$-integral over $\Gamma'_y$.

It suffices to prove that $J_{b,b}(i,j) = \ind{i=j, i \leq n}$, $J_{a,b} = J_1$, $J_{b,a} = - J_2$ and $J_{a,a} = J_3$.
Observe that $J_{a,a}$ equals $J_3$ by definition and \eqref{eqn:ALB}.

\paragraph{\textbf{Proof that $J_{b,b}(i,j) = \ind{i=j,\, i \leq n}$}.}
In $J_{b,b}(i,j)$ all contours are around the poles $1/b_k$.
Arrange the contour so that the $\zeta$-contour contains the $\omega$-contour.
So the ordering of the contours makes $|z| > |w| > |\zeta| > |\omega|$.

Contract the $w$-contour with residue at $w = \omega$. Then,
\begin{align*}
	J_{b,b}(i,j) = \frac{1}{(2 \pi \mathbold{i})^3}\oint_{\gamma_b} d \zeta \oint_{\gamma_b} d\omega \oint_{\Gamma_b} dz \,
	\frac{G(z \vt [n], [m], h-1)}{G(\zeta \vt [i], [m], h-1) (z-\omega)(z-\zeta)}
	\frac{\prod_{k=n+1}^N (\omega - 1/b_k)}{\prod_{k=j}^N (\omega - 1/b_k)}.
\end{align*}
When $j > n$ there is no $\omega$-pole and the integral is zero. So assume $j \leq n$.

Contract the $z$-contour with residues at $z = \zeta$ and then $z = \omega$. This gives $J_{b,b}(i,j) = (I) + (II)$ with
\begin{align*}
(I) &= \ind{j \leq n} \frac{1}{(2 \pi \mathbold{i})^2}\oint_{\gamma_b} d \zeta \oint_{\gamma_b} d\omega \, 
\frac{\prod_{k=1}^n (\zeta - 1/b_k)}{\prod_{k=1}^i (\zeta - 1/b_k) \prod_{k=j}^n(\omega - 1/b_k)\, (\zeta - \omega)} \\
(II) & = \ind{j \leq n} \frac{1}{(2 \pi \mathbold{i})^2}\oint_{\gamma_b} d \zeta \oint_{\gamma_b} d\omega \, 
\frac{ G(\omega \vt [j-1], [m], h-1)}{G(\zeta \vt [i], [m], h-1)\, (\omega - \zeta)}
\end{align*}
Term (II) is zero because the $\omega$-contour can be contracted.
Term (I) is zero when $i > n$ because the $\zeta$-contour can be contracted to $\infty$.
For $i \leq n$, recalling that $|\zeta| > |\omega|$, contract the $\zeta$-contour with residue at $\zeta = \omega$.
So,
\begin{align*}
	J_{b,b}(i,j) &= \ind{i,j \leq n} \frac{1}{(2 \pi \mathbold{i})} \oint_{\gamma_b} d\omega \,
	\frac{\prod_{k=i+1}^n (\omega - 1/b_k)}{\prod_{k=j}^n (\omega - 1/b_k)}. 
\end{align*}

In the integral above, when $i=j$, $J_{b,b}(i,i) = \ind{i \leq n}$ as there is a simple pole at $1/b_i$ with residue 1.
When $i < j$ the integral is zero because the contour can be contracted to 0. When $i > j$ the integral is also
zero as the contour can be contracted to $\infty$. This shows $J_{b,b}(i,j) = \ind{i=j, i \leq n}$.

\paragraph{\textbf{Proof that $J_{a,b} = J_1 $}.} 
Contract the $w$-contour with a residue at $w = \omega$, which gives
\begin{align*}
	J_{a,b}(i,j) &= \frac{1}{(2 \pi \mathbold{i})^3} \oint_{\gamma_b} d \zeta\, \oint_{\gamma_b} d \omega \oint_{\Gamma_a} dz\, \\
	& \frac{G(z \vt [n], [m], h-1)}{(z - \zeta)(z-\omega) \, G(\zeta \vt [i], [m], h-1)}
	\frac{\prod_{k=n+1}^{N} (\omega - 1/b_k)}{\prod_{k=j}^N (\omega - 1/b_k)}.
\end{align*}
The $\omega$-contour can be contracted if $j > n$. So assume $j \leq n$ and contract the $\omega$-contour to $\infty$.
This incurs a residue at $\omega = z$, resulting in $J_{a,b} = J_1$.

\paragraph{\textbf{Proof that $J_{b,a} = - J_2 $}.}
Contract the $z$-contour with a residue at $z = \zeta$, which gives
\begin{align*}
J_{b,a}(i,j) &= \frac{1}{(2 \pi \mathbold{i})^3} \oint_{\gamma_b} d \zeta\, \oint_{\gamma_b} d \omega \oint_{\Gamma'_a} dw\, \\
& \frac{G(w \vt [N]\setminus [n], [M]\setminus [m], H-h)}{(\zeta - w)(w-\omega) \, G(\omega \vt [N]\setminus [j-1], [M]\setminus [m], H-h)}
\frac{\prod_{k=1}^n (\zeta - 1/b_k)}{\prod_{k=1}^i (\zeta - 1/b_k)}.
\end{align*}
The $\zeta$-contour can be contracted if $i \leq n$. So assume $i > n$ and contract the $\zeta$-contour to $\infty$. 
This incurs a residue at $\zeta = w$, to give $J_{b,a} = -J_2$.
\end{proof}

\subsection{Exponential last passage percolation} \label{sec:5}
The discrete polynuclear growth model becomes the exponential last passage percolation model under suitable re-scaling of the weights $\omega_{i,j}$.
For $\eps > 0$, write $a_i = 1 - \eps \alpha_i$ and $b_j = 1- \eps \beta_j$ for a new set of parameters $\alpha_i, \beta_j \geq 0$.
The random variable $\eps w_{i,j}$ converges in distribution as $\eps$ tends to zero to an exponential random variable of rate $\alpha_i + \beta_j$.
In this limit we find analogues of the previous formulas for the exponential model. These are stated in the following.

Consider independent exponential weights $\omega_{i,j} \sim \mathrm{Exp}(\alpha_i + \beta_j)$ (rate $\alpha_i + \beta_j$) with
$\alpha_i + \beta_j > 0$. Let $\Gb$ be the growth function defined by \eqref{eqn:G} in terms of
these exponential weights $\omega_{i,j}$. Define $\vec{\Gb}$ as in \eqref{eqn:Gvec}, which now takes values in
$$\W_N(\R) = \{ z \in \R^N: z_1 \leq z_2 \leq \cdots \leq z_N \}.$$

\begin{cor} \label{cor:exp1}
	For $x,y \in \W_N(\R)$, the transition density matrix of $\vec{\Gb}$ is given by
	$$\pr{\vec{\Gb}(m) \in dy \vt \vec{\Gb}(0) = x} = \dt{M_{\mathrm{Exp}}(i,j \vt x,y)}_{i,j} dy$$
	where
	$$M_{\mathrm{Exp}}(i,j \vt x,y) = \intz{R} dz\, e^{y_j(z-\beta_j) - x_i (z-\beta_i)} \frac{\prod_{k=1}^j (z-\beta_k)}{\prod_{k=1}^i (z-\beta_k)} \prod_{k=1}^m \frac{\alpha_k+\beta_j}{z+\alpha_k}$$
	and $R > \max_{k,\ell} \{\alpha_k, \beta_{\ell} \}$.
\end{cor}

In order to obtain this corollary from Theorem \ref{thm:1} one writes $a_i = 1 - \eps \alpha_i$ and  $b_j = 1 - \eps \beta_j$,
sets $x^{\eps}_i = \lfloor \eps^{-1}x_i \rfloor $ and $y_i^{\eps} = \lfloor \eps^{-1} y_i \rfloor$, and then considers the limit as
$\eps \to 0$ of $M(i,j \vt x^{\eps}, y^{\eps})$. The limit is obtained by changing variable $z \mapsto 1 + \eps z$ in
the contour integral defining $M(i,j \vt x^{\eps},y^{\eps})$. The following corollary is obtained in the same manner.

\begin{cor} \label{cor:exp2}
	Given $x \in \W_N(\R)$ and $0=n_0 < n_1 < n_2 < \cdots < n_p = N$, the distribution function
	$$\pr{G(m,n_1) \leq h_1, \ldots, G(m,n_p) \leq h_p \vt \vec{\Gb}(0) = x} = \dt{F_{\mathrm{Exp}}(i,j \vt x)}_{i,j}$$
	where the matrix
	$$ F_{\mathrm{Exp}}(i,j \vt x) = \intz{R} dz\, e^{h(j)(z-\beta_j) - x_i(z-\beta_i)} \frac{\prod_{k=1}^{j-1} (z-\beta_k)}{\prod_{k=1}^i (z-\beta_k)} \prod_{k=1}^m \frac{\alpha_k+\beta_j}{z+\alpha_k}.$$
	Here $h(j) =h_k$ if $j \in (n_{k-1}, n_k]$ and $R > \max_{k,j} \{\alpha_k, \beta_j \}$.
\end{cor}
A Fredholm determinant formula when $x =0$ may be obtained from Proposition \ref{prop:spatialFredholm}.

Finally, the following expresses the two-time distribution in exponential last passage percolation as a corollary of Theorem \ref{thm:twotime}.
A Fredholm determinant formula when $x=0$ can be derived from Theorem \ref{thm:4}.
\begin{cor} \label{cor:exp3}
	The two-time distribution function of the exponential last passage percolation model is given by
	$$ \pr{\Gb(m,n) < h, \Gb(M,N) < H \vt \vec{\Gb}(0)=x} = \frac{1}{2 \pi \mathbold{i}} \oint \limits_{|\theta| = r} d \theta\,
	\frac{\dt{\theta^{\ind{i > n}}L^{\rm{Exp}}_1 - \theta^{-\ind{i \leq n}}L^{\rm{Exp}}_2}}{\theta -1}$$
	where the radius $r > 1$ and $L^{\rm{Exp}}_1, L^{\rm{Exp}}_2$ are the following $N \times N$ matrices.
	\begin{align*}
	L^{\rm{Exp}}_1(i,j) &= \frac{1}{(2 \pi \mathbold{i})^2} \oint \limits_{|z|=R_1} dz \oint \limits_{|w|=R_2} dw\,
	\frac{e^{(z-\beta_i)h + (w-\beta_j)(H-h) - x_i(z-\beta_i)}}{(z-w)} \\
	& \times  \frac{\prod_{k=1}^n(z-\beta_k)}{\prod_{k=1}^i(z-\beta_k)} \frac{\prod_{k=1}^{j-1}(w-\beta_k)}{\prod_{k=1}^n(w-\beta_k)}
	\prod_{k=1}^m \frac{\alpha_k+\beta_i}{z+\alpha_k} \prod_{k= m+1}^M \frac{\alpha_k+\beta_j}{w+\alpha_k}.
	\end{align*}
	The contours are arranged so that $R_1 > R_2 > \max_{k,j} \{\alpha_k, \beta_j\}$.
	The matrix $L^{\rm{Exp}}_2$ looks the same except the ordering of contours is reversed to $R_2 > R_1 > \max_{k,j} \{\alpha_k, \beta_j\}$.
\end{cor}

\section{Asymptotics and scaling limits} \label{partIII}

\subsection{Baik-Ben\,Arous-P\'ech\'e distribution at two times} \label{sec:7}

Consider the two-time distribution of inhomogeneous model with every $b_j = 1$, $a_1, \ldots, a_r$ being variable, and every $a_i = q$ for $i > r$.
The matrix $F(\theta)$ from Theorem \ref{thm:4} is a finite rank perturbation of the corresponding matrix for the
homogeneous model where every $a_i = q$. We are interested in the two-time distribution of this instance in the KPZ-scaling limit, which is the following.
Its single time scaling limit, known as the Baik-Ben\,Arous-P\'ech\'e distribution, has been studied in \cite{BBP}.

Define constants $c_0, c_1, \ldots, c_4$ by
\begin{align}\label{scalingconstants}
& c_0=q^{-\frac{1}{3}}(1+\sqrt{q})^{\frac{1}{3}},\quad c_1=q^{-\frac{1}{6}}(1+\sqrt{q})^{\frac{2}{3}},\quad c_2=\frac{2\sqrt{q}}{1-\sqrt{q}},\\ \nonumber
& c_3=\frac{q^{\frac{1}{6}}(1+\sqrt{q})^{\frac{1}{3}}}{1-\sqrt{q}}, \quad c_4 = \frac{q^{1/3}(1-\sqrt{q})}{(1+\sqrt{q})^{1/3}}.
\end{align}

For a large parameter $T$, write $n,N, m, M, h$ and $H$ according to the following scaling. Consider
temporal parameters $0 < t_1 < t_2$, spatial parameters $x_1, x_2 \in \R$ and $\xi_1, \xi_2 \in \R$.
For a choice of these, set (ignoring rounding)
\begin{align} \label{kpzscaling}
& n = t_1 T - c_1 x_1 (t_1T)^{2/3} & N = t_2 T - c_1 x_2(t_2T)^{2/3} \\ \nonumber
& m = t_1 T + c_1 x_1 (t_1T)^{2/3} & M = t_2 T + c_1 x_2(t_2T)^{2/3} \\ \nonumber
& h = c_2 (t_1 T) + c_3 \xi_1(t_1T)^{1/3} & H = c_2 (t_2 T) + c_3 \xi_2(t_2 T)^{1/3} .
\end{align}

Introduce the notation $\D n = N-n$, $\D m = M-m$ and $\D h = H-h$. If we set
\begin{equation} \label{eqn:deltas}
\D t = t_2 - t_1, \quad \D x = (\frac{t_2}{\D t})^{2/3} x_2 - (\frac{t_1}{\D t})^{2/3}x_1, \quad \D \xi = (\frac{t_2}{\D t})^{1/3} \xi_2 - (\frac{t_1}{\D t})^{1/3}\xi_1,
\end{equation}
then it holds that $\D n = \D t T - c_1 \D x(\D tT)^{2/3}$ and likewise for $\D m$ and $\D h$.

The parameters $a_1, \ldots, a_r$ are scaled according to
\begin{equation} \label{ascaling}
a_k = \sqrt{q} - \frac{c_4}{T^{1/3}}\cdot \lambda_k \quad \text{with}\; \lambda_k > 0.
\end{equation}
The $\lambda_k$ are parameters. Assume that $r < m$ and that $r$ remains fixed, independently of $T$.

We want to consider the large $T$ limit of the two-time distribution under this scaling of the parameters.
The limit is represented by a contour integral of a Fredholm determinant over $L^2(\R)$.
We will build up to it in the coming sections.

\subsubsection{Statement of the limit theorem}
Define the function
\begin{equation} \label{eqn:Glimit}
\G(z \vt t, x, \xi) = \exp \left \{ \frac{t}{3}z^3 + t^{2/3}x z^2 - t^{1/3} \xi z \right \}
\end{equation}
for $z \in \C$, $t > 0$ and $x, \xi \in \R$.

Let $d_1, d_2, D_1$ and $D_2$ be positive real numbers such that
$$ D_1, D_2 < \min \, \{\lambda_1, \lambda_2, \ldots, \lambda_r \}, $$
where every $\lambda_k > 0$ (the same as in \eqref{ascaling}).
Denote by $\Re(z) = d$ the vertical line crossing the real axis at $d$ and oriented upwards.
Let $\mu$ be a sufficiently large scalar that will be used in a conjugation factor.

Define kernels $\J_{1}, \J_{2}, \J_{3,<}$ and $\J_{3,>}$ over $L^2(\R)$ as follows.
They depend implicitly on the $\lambda_k$s.
\begin{align} \label{eqn:Js}
& \J_{1}(u,v) = e^{\mu(v-u)} \, \ind{v \leq 0}\,  \frac{1}{(2 \pi \mathbold{i})^2}
\oint \limits_{\Re(\zeta) = -d_1} d \zeta \oint \limits_{\Re(z) = D_1} dz\, 
\frac{\G(z \vt t_1, x_1, \xi_1) e^{zv-\zeta u}}{\G(\zeta \vt t_1, x_1, \xi_1) (z-\zeta)}
\prod_{k=1}^r \frac{\lambda_k-\zeta}{\lambda_k-z} \\ \nonumber
& \J_{2}(u,v) = e^{\mu(v-u)} \, \ind{u > 0}\,  \frac{1}{(2 \pi \mathbold{i})^2}
\oint \limits_{\Re(\omega) = -d_1} d \omega \oint \limits_{\Re(w) = D_1} dw\, 
\frac{\G(w \vt \D t, \D x, \D \xi) e^{\omega v - wu}}{\G(\omega \vt \D t, \D x, \D \xi) (w-\omega)} \\ \nonumber
& \J_{3, s}(u,v) =  \; e^{\mu(v-u)} \, \frac{1}{(2 \pi \mathbold{i})^4}
\oint \limits_{\Re(\zeta) = -d_1} d \zeta \oint \limits_{\Re(\omega) = -d_2} d \omega 
\oint \limits_{\Re(z) = D_1} dz  \oint \limits_{\Re(w) = D_2} dw \\ \nonumber
& \qquad \frac{\G(z \vt t_1, x_1, \xi_1) \, \G(w \vt \D t, \D x, \D \xi)}{\G(\zeta \vt t_1, x_1, \xi_1) \, \G(\omega \vt \D t, \D x, \D \xi)} \cdot
\frac{e^{\omega v - \zeta u} }{(z-\zeta)(w-\omega)(z-w)} \prod_{k=1}^r \frac{\lambda_k - \zeta}{\lambda_k - z}.
\end{align}
If $s$ equals $<$ then $D_1 < D_2$, that is, the $z$-contour is to the left of the $w$-contour. If $s$ equals $>$ then $D_1 > D_2$,
so that the ordering of the contours is reversed.

The kernels are of trace class if $\mu$ is sufficiently large in terms of $x_1,x_2,t_1$ and $t_2$ because then their absolute
values are bounded by terms of the form $e^{-\mu' u} \Ai(-u) e^{\mu' v} \Ai(v)$ where $\Ai()$ is the Airy function.

The following theorem will be proved by doing a saddle point analysis of the matrices from Theorem \ref{thm:4}.
\begin{thm} \label{thm:5}
	Consider the two-time distribution $\pr{G(m,n) < h, G(M,N) < H}$ for the inhomogeneous growth model
	where every $b_j = 1$, $a_i = q$ for $i > r$ and $a_1, \ldots, a_r$ are according to \eqref{ascaling}.
	Assume that $n,m,h,N,M,H$ are given by \eqref{kpzscaling}. Then in the limit at $T$ tends to infinity,
	the two time distribution functions converges to
	$$ \frac{1}{2 \pi \mathbold{i}} \oint \limits_{|\theta| = r} \frac{\dt{I + F_{\lambda}(\theta)}_{L^2(\R)}}{\theta - 1}$$
	where $r > 1$ and
	$$F_{\lambda}(\theta)(u,v) = \theta^{\ind{u > 0}} F_{1, \lambda}(u,v) + \theta^{- \ind{u \leq 0}} F_{2, \lambda}(u,v).$$
	The kernels $F_{1, \lambda}$ and $F_{2, \lambda}$ are given by
	\begin{equation*}
	F_{1, \lambda} = \J_{2} - \J_{1}  + \J_{3,<} \quad \text{and} \quad F_{2, \lambda} = \J_{1} - \J_{2} - \J_{3,>}\,.
	\end{equation*}
\end{thm}
We do not prove that the limit defines a probability distribution function in the parameters $\xi_1$ and $\xi_2$, but it is not hard to show
it is the case based on the fact that the corresponding single time distribution is such (see \cite{BBP}).
Remark also that if $F_1$ is thought of in terms of the parameters $t_1, x_1, \xi_1, \D t, \D x$ and $\D \xi$, that is, as
$F_1(u,v \vt t_1, x_1, \xi_1, \D t, \D x, \D \xi)$, then $F_2(u,v) = F_1(-v, -u \vt \D t, \D x, \D \xi, t_1, x_1, \xi_1)$.

\subsubsection{Preparation for the proof}
This section describes how to embed matrices into $L^2(\R)$ for the sake of doing asymptotic analysis.
We will also define contours for saddle point analysis and re-express the $J$ matrices from Theorem \ref{thm:4} for asymptotics.
\smallskip

\paragraph{\textbf{Embedding.}} Embed an $N \times N$ matrix $M$ (where $n$ and $N$ are the parameters from the two-time distribution)
as a kernel over $L^2(\R)$ by the formula
$$M \mapsto F(u,v) = M(n + \lceil u \rceil, n + \lceil v \rceil)$$
where $u, v \in \R$. Set $F(u,v)$ to be zero when $n + \lceil u \rceil$ or $n + \lceil v \rceil$ lie outside the set $[N]$.
According to this embedding, $F$ takes the value $M(i,j)$ over the unit square $(i-n-1, i-n] \times (j-n-1,j-n]$. 
Then it follows readily that
$$\dt{I + M}_{N \times N} = \dt{I + F}_{L^2(\R)},$$
where the latter determinant should be taken as the Fredholm series expansion of $F$.
The KPZ re-scaled kernel is defined to be
\begin{equation} \label{eqn:Fkpz}
F_T(u,v) = \nu_T \cdot F(\nu_T u, \, \nu_T v) \quad \text{with}\; \nu_T = c_0 T^{1/3}.
\end{equation}
Note that $\dt{I+F}_{L^2(\R)}$ equals $\dt{I+F_T}_{L^2(\R)}$, which follows from re-scaling variables in the Fredholm series expansion. The matrices $F_1$ and $F_2$ from Theorem \ref{thm:4} will be considered under the scaling \eqref{eqn:Fkpz}.
\smallskip

\paragraph{\textbf{Descent contours.}}
Consider circular contours $\gamma_0$, around 0, and $\gamma_1$, around 1, as contours for the integration variables $\zeta, z, \omega, w$.
First, define
\begin{equation} \label{eqn:wc}
w_c = 1- \sqrt{q},
\end{equation}
which is the critical point around which asymptotics will be performed. Now, for a (large) parameter $K$, define
\begin{align} \label{eqn:contours}
\gamma_0 &= \gamma_0(\sigma, d) = w_c(1- \frac{d}{K^{1/3}}) e^{\mathbold{i} \sigma K^{-1/3}} \quad |\sigma| \leq \pi K^{1/3}, \\ \nonumber
\gamma_1 &= \gamma_1(\sigma, d) = 1 - \sqrt{q}(1- \frac{d}{K^{1/3}}) e^{\mathbold{i} \sigma K^{-1/3}} \quad |\sigma| \leq \pi K^{1/3}.
\end{align}
The parameter $d$ should satisfy $0 < d < K^{1/3}$. Observe that if $\sigma$ remains bounded independently of $K$ then
one has the expansions
$$ \gamma_0(\sigma, d) = w_c + w_c\frac{(\mathbold{i} \sigma - d)}{K^{1/3}} + O(K^{-2/3}), \quad
\gamma_1(\sigma, d) = w_c + \sqrt{q} \frac{(-\mathbold{i}\sigma + d)}{K^{1/3}} + O(K^{-2/3}).$$
So, locally around $\sigma = 0$, the contours are vertical lines.
\smallskip

\paragraph{\textbf{Re-expressing kernels from Theorem \ref{thm:4}.}}
Consider the $G$-function and the $J$-matrices from Theorem \ref{thm:4}.
Changing variables $x \mapsto 1-x$ for $x = z, w, \zeta, \omega$ in the integrals, and substituting $b_j =1$ and $a_i = q$ for $i > r$,
the formula becomes as follows.

Define
\begin{equation} \label{eqn:Gstar}
G^{*}(z \vt n,m,h) = \frac{z^n (1-z)^{m+h}}{\left (1- \frac{z}{1-q} \right )^{m}} \cdot
\frac{\left (1- \frac{w_c}{1-q} \right )^{m}}{w_c^n (1-w_c)^{m+h}}.
\end{equation}
The asymptotics will involve $G^{*}$, which is normalized around the critical point $w_c$.

Expressing the $J$ matrices in terms of $G^{*}$ gives the following. The factor $w_c^{j-i}$ below corresponds
to a conjugation, and since it appears in front of every $J$-matrix, it can be removed from the determinant as we
will do in the next section. The contours below should not intersect.
\begin{align} \label{eqn:J10}
&  J_{1}(i,j) = \ind{j \leq n} \, \frac{w_c^{j-i-1}}{(2 \pi \mathbold{i})^2}\oint_{|\zeta| = r_1} d \zeta \, \oint_{|z-1| = \rho_1} dz \,
\frac{G^{*}(z \vt j-1, m-r, h)}{G^{*}(\zeta \vt i, m-r, h)\, (z-\zeta)} \prod_{k=1}^r \frac{1-a_k-\zeta}{1-a_k-z} \big (\frac{1-z}{1-\zeta} \big )^r.
\end{align}
\begin{align} \label{eqn:J01}
	&J_{2}(i,j) = \ind{i > n} \, \frac{w_c^{j-i-1}}{(2 \pi \mathbold{i})^2} \oint_{|\omega| = r_1} d\omega \, \oint_{|w-1| = \rho_2} dw \,
	\frac{G^{*}(w \vt N-i, \D m, \D h)}{G^{*}(\omega \vt N+1-j, \D  m, \D h) \, (w-\omega)} 
\end{align}
\begin{align} \label{eqn:J11}
& J_{3}(i,j) = \frac{w_c^{j-i-1}}{ (2 \pi \mathbold{i})^4} \oint_{|\zeta|=r_1} d\zeta \, \oint_{|\omega|=r_2} d \omega \, \oint_{|z-1|=\rho_1} dz\,  \oint_{|w-1| = \rho_2} dw  \\ \nonumber
& \frac{G^{*}(z \vt n,m-r, h) G^{*}(w \vt \D n, \D m, \D h)}{G^{*}(\zeta \vt i,m-r, h) G^{*}(\omega \vt N+1-j, \D m, \D h)\,(z-\zeta)(w-\omega)(z-w)} \prod_{k=1}^r \frac{1-a_k-\zeta}{1-a_k-z} \big (\frac{1-z}{1-\zeta} \big )^r.
\end{align}
The radii $\rho_1 > \rho_2 > \max_k \{a_k\}$ and $r_k < 1-\rho_k$.
The matrix $J_{4}$ is the same as $J_{3}$ except that the contours are ordered to satisfy $\rho_1 < \rho_2$.

\subsubsection{Proof of the theorem} 

In order to derive the limiting formula we need to show that the matrices $F_1$ and $F_2$ from Theorem \ref{thm:4}, under KPZ-scaling
\eqref{eqn:Fkpz} with all the parameters scaled according to \eqref{kpzscaling}, converge to the corresponding matrices $F_1$ and $F_2$
from Theorem \ref{thm:5}. In order to do so, it suffices to show convergence of each of the matrices
$J_{1}, J_{2}, J_{3}$ and $J_{4}$.

The mode of convergence required is one for which the determinant of $I + F(\theta)$ converges to the Fredholm determinant of the limit kernel.
This will happen if we show the following two things.
\begin{enumerate}
	\item Prove that if the kernel variables $u$ and $v$ remain bounded then the KPZ re-scaled kernels of each of the $J$ matrices
	converge to the corresponding $\J$ matrices.
	\item Establish decay estimates of the form $|F_T(u,v)| \leq g_1(u) g_2(v)$ for each of the KPZ re-scaled kernels, where
	$g_1$ and $g_2$ are bounded and integrable functions over $\R$. One can then use the dominated convergence theorem and
	Hadamard's inequality to conclude that the Fredholm determinant of the re-scaled kernels converge to the Fredholm
	determinant of their limit.
\end{enumerate}

We will carry out the procedures above for the matrix $J_{3}$ in order to show that it converges to $\J_{3, <}$.
The steps are entirely alike for the other $J$ matrices; $J_{1}$ converges to $-\J_{1}$, $J_{2}$
converges to $-\J_{2}$ and $J_{4}$ converges to $\J_{3, >}$.
We will omit these for brevity.

In the following it is assumed that $x_1, x_2, \xi_1, \xi_2 \in \R$ are fixed parameters, as well as $0 < t_1 < t_2$ are fixed.
Suppose all these parameters are at most $L$ in absolute value. We will denote by $C_{q,L}$ a constant
that depends only on $q$ and $L$, but whose value may change from place to place.
\smallskip

\paragraph{\textbf{Point-wise limit of $J_{3}$}.}
Under KPZ scaling, the indices $i$ and $j$ are written as $i = n + \lceil \nu_T u \rceil$ and $j = n + \lceil \nu_T v \rceil$
for $u, v \in \R$ and $\nu_T = c_0 T^{1/3}$. The KPZ rescaled kernel for $J_{3}$ is
$$ J_T(u,v) = \nu_T \cdot J_{1,1}(n + \lceil \nu_T u \rceil, n + \lceil \nu_T v \rceil ).$$
Consider $J_{1,1}$ in the form given by \eqref{eqn:J11} and ignore the conjugation factor $w_c^{j-i}$.

First, we choose contours for the variables $\zeta, \omega, z, w$ in the four-fold contour integral \eqref{eqn:J11}.
Recall the contours from \eqref{eqn:contours}.
\begin{align*}
& \zeta = \zeta(\sigma_1) \in \gamma_0 \left ( \frac{c_4}{w_c} \sigma_1, \frac{c_4}{w_c} d_1 \right)
& z = z(\sigma_2) \in \gamma_1\left ( \frac{c_4}{\sqrt{q}} \sigma_1, \frac{c_4}{\sqrt{q}} D_1 \right) 
& \quad K = t_1T, \\
& \omega = \omega(\sigma_3) \in \gamma_0 \left ( \frac{c_4}{w_c} \sigma_3, \frac{c_4}{w_c} d_2 \right)
& w = w(\sigma_4) \in \gamma_1\left ( \frac{c_4}{\sqrt{q}} \sigma_4, \frac{c_4}{\sqrt{q}} D_2 \right)
& \quad K =\D t T.
\end{align*}
We need to have $D_1/t_1^{1/3} < D_2/(\D t)^{1/3} < \min_k \{ \lambda_k \}$ in order to satisfy the constraint $\rho_1 > \rho_2$
and to ensure that the poles at $1-a_k$ lie within the contours.

Due to the choice of contours, Lemma 5.3 of \cite{JR} (which gives decay estimates for $G^{*}$ along these contours)
and some simple bookkeeping leads to the estimate
$$ \nu_T \cdot \left ( \text{integrand of}\; J_{1,1}(n + \lceil \nu_T u \rceil, n + \lceil \nu_T v \rceil) \right )
\leq C_1 e^{- C_2 (\sigma_1^2+\sigma_2^2+\sigma_3^2+\sigma_4^2)},$$
so long as $u$ and $v$ remain bounded and where $C_1$ and $C_2$ are constants that depend on $u,v$ and the parameters $t_i,x_i,\xi_i$. Note the $\sigma_k$s are variables of integration. This allows us to use the dominated convergence theorem to
find the limiting integral for $J_T(u,v)$ by considering its point-wise limit with $u$, $v$ and the $\sigma_k$s held fixed.

If the variables $\sigma_k$ are kept fixed, one has by Taylor expansion that
\begin{align*}
& \zeta(\sigma_1) = w_c + \frac{c_4}{(t_1T)^{1/3}} (\mathbold{i} \sigma_1 - d_1) + C_{q,L}T^{-\frac{2}{3}}
& z(\sigma_2) = w_c + \frac{c_4}{(t_1T)^{1/3}} (\mathbold{i} \sigma_2 + D_1) + C_{q,L}T^{-\frac{2}{3}} \\
& \omega(\sigma_3) = w_c + \frac{c_4}{(\D t T)^{1/3}} (\mathbold{i} \sigma_3 - d_2) + C_{q,L}T^{-\frac{2}{3}}
& w(\sigma_2) = w_c + \frac{c_4}{(\D t T)^{1/3}} (\mathbold{i} \sigma_4 + D_2) + C_{q,L}T^{-\frac{2}{3}}.
\end{align*}
Write
$$\zeta' = (\mathbold{i}\sigma_1 -d_1)/ t_1^{1/3}, z' = (\mathbold{i}\sigma_2 + D_1)/t_1^{1/3},
\omega' = (\mathbold{i}\sigma_3-d_2)/(\D t)^{1/3}, w' = (\mathbold{i} \sigma_4 + D_2)/(\D t)^{1/3}.$$
In these new variables, at $T$ tends to infinity, the contours become
vertical lines. The contours of $\zeta'$ and $\omega'$ become, respectively, the lines $\Re(\zeta') = -d_1/t_1^{1/3}$
and $\Re(\omega') = -d_2/(\D t)^{1/3}$, oriented upwards. The $z$-contour becomes $\Re(z') = D_1/t_1^{1/3}$, oriented downwards.
The $w$-contour becomes downwardly oriented $\Re(w') = D_2/(\D t)^{1/3}$. If these contours are then oriented upwards,
we obtain a factor of $(-1)^2 = 1$.

Set $d'_1 = d_1/t^{1/3}$, $d'_2 = d_2/(\D t)^{1/3}$ and $D'_1 = D_1/t_1^{1/3}$, $D'_2 = D_2/(\D t)^{1/3}$.
The constraints on the limiting contours become $d'_1, d'_2 > 0$ and $0 < D'_1 < D'_2 < \min_k \{ \lambda_k\}$.

Having found the limit contours, we consider the behaviour of the integrand along these contours.
By Lemma 5.2 of \cite{JR} (which stipulates the local behaviour of $G^{*}$ around $w_c$ under KPZ scaling), if
$$n = K - c_1 x K^{2/3} + c_0 u K^{1/3}, \, m = K + c_1 xK^{2/3}, \, h = c_2K + c_3 \xi K^{1/3}$$
and $w = w_c + (c_4/K^{1/3}) w'$, then uniformly for $w'$ in any compact set, 
$$\lim_{K \to \infty} G^{*}(w \vt n,m,h) = \G(w' \vt 1, x, \xi - u) = \exp \{ (w')^3/3 + x (w')^2 - (\xi-u)w'\}.$$
Consequently, as $T \to \infty$, one has that (recall \eqref{eqn:Glimit}):
\begin{align*}
& G^{*}(z \vt n, m-r, h) \to \G(t_1^{1/3}z' \vt 1,x_1, \xi_1) = \G(z' \vt t_1, x_1, \xi_1), \\
& G^{*}(w \vt \D n, \D m, \D h) \to \G((\D t)^{1/3}w' \vt 1, \D x, \D \xi) = \G(w' \vt \D t, \D x, \D \xi),\\
& G^{*}(\zeta \vt i, m-r, h) \to \G(t_1^{1/3}\zeta' \vt 1,x_1, \xi_1 - t_1^{-1/3}u) = \G(\zeta' \vt t_1, x_1, \xi_1)e^{\zeta' u}, \\
& G^{*}(\omega \vt N+1-j, \D m, \D h) \to \G((\D t)^{1/3}\omega' \vt 1, \D x, \D \xi + (\D t)^{-1/3}v) = \G(\omega' \vt \D t, \D x, \D \xi)e^{-\omega' v}.
\end{align*}

Next, it is easy to see from a calculation that
$$\frac{\nu_T}{w_c}\, \frac{d \zeta \, d \omega \, d z \, d w}{(z-\zeta)(w-\omega)(z-w)} =
\frac{d \zeta' \, d \omega' \, d z' \, d w'}{(z'-\zeta')(w'-\omega')(z'-w')} + C_{q,L}T^{-1/3}.$$
Also, $(1-z)/(1-\zeta)$ tends to 1 and so does its $r$-th power.

Finally, consider the product $\prod_{k=1}^r \frac{1-a_k-\zeta}{1-a_k-z}$. Observe that
$$ \frac{1-a_k-\zeta}{1-a_k-z} = \frac{1-a_k - w_c - c_4 \zeta' T^{-1/3}}{1-a_k - w_c - c_4 z' T^{-1/3}} = \frac{\lambda_k - \zeta'}{\lambda_k -z'}.$$
So the product converges to $\prod_{k=1}^r \frac{\lambda_k-\zeta'}{\lambda_k-z'}$.

Putting it together, we find that if $u$ and $v$ remain bounded then $J_T(u,v)$ converges to
\begin{align*}
& \frac{1}{(2 \pi \mathbold{i})^4} \oint \limits_{\Re(\zeta') = -d'_1} d \zeta' \oint \limits_{\Re(\omega') = -d'_2} d \omega'
\oint \limits_{\Re(z') = D'_1} dz'  \oint \limits_{\Re(w') = D'_2} dw' \\
& \frac{\G(z' \vt t_1, x_1, \xi_1) \, \G(w' \vt \D t, \D x, \D \xi)}{\G(\zeta' \vt t_1, x_1, \xi_1) \, \G(\omega' \vt \D t, \D x, \D \xi)} \cdot
\frac{e^{\omega' v - \zeta' u} }{(z'-\zeta')(w'-\omega')(z'-w')} \prod_{k=1}^r \frac{\lambda_k - \zeta'}{\lambda_k - z'}.
\end{align*}
The constraint on the contours is that $d'_1, d'_2 > 0$ and $ 0 < D'_1 < D'_2 < \min_k \{\lambda_k\}$.
This limit is precisely $\J_{1,1, <}$ but without the conjugation factor $e^{\mu (v-u)}$.
\smallskip

\paragraph{\textbf{Decay estimate for $J_{3}$.}}
In order to have the decay estimate on $J_{3}$ for step (2) of the limit argument, one has to include the
conjugation factor $e^{\mu (v-u)}$ for a sufficiently large constant $\mu$ in front of
the KPZ re-scaled kernel $J_T(u,v)$ of $J_{3}$.

First, recall the choice of contours for the variables $\zeta, \omega, z, w$ from the previous step.
By Lemma 5.3 of \cite{JR}, one has the following estimates where $C_1$ and $C_2$ are constants
that depend only on $q$ and $L$ (recall all parameters $t_i, x_i$ and $\xi_i$  are bounded by $L$).
\begin{align*}
& |G^{*}(\zeta(\sigma_1) \vt i, m-r, h)|^{-1} \leq C_1 e^{-C_2 \sigma_1^2 + \Psi(u)}\\
& |G^{*}(\omega(\sigma_3) \vt N+1-j, \D m, \D h)|^{-1} \leq C_1 e^{-C_2 \sigma_3^2 + \Psi(-v)} \\
& |G^{*}(z(\sigma_2) \vt n, m-r, h)| \leq C_1 e^{-C_2 \sigma_2^2} \\
& |G^{*}(w(\sigma_4) \vt \D n, \D m, \D h)| \leq C_1 e^{-C_2 \sigma_4^2}.
\end{align*}
Here $\Psi(x) = - \mu_1 (x)_{-}^{3/2} + \mu_2(x)_{+}$ for some positive constants $\mu_1$ and $\mu_2$.

It is easy to see that there is a constant $C_3$ that depends only on $q$ and the $\lambda_k$s such that
$$ \left | \prod_{k=1}^r \frac{1-a_k-\zeta}{1- a_k-z} \cdot \big (\frac{1-z}{1-\zeta}\big )^r \right | \leq C_3.$$

It follows from these estimates that for the KPZ re-scaled kernel $J_T$,
$$e^{\mu(v-u)} |J_T(u,v)| \leq C_{q,L, \lambda} \, e^{-\mu u + \Psi(u)} \cdot e^{\mu v + \Psi(-v)}.$$
Finally, observe that for $\mu > \max \{\mu_1, \mu_2 \}$, the function $e^{-\mu x + \Psi(x)}$
is bounded and integrable. This shows the decay estimate required for the second step, and completes the argument.

\subsection{Two-time distribution of the KPZ fixed point started from one-sided Brownian motion} \label{sec:8}
This section considers the model from the previous section when $r=1$ and in the limit $\lambda_1 \to 0$.
As mentioned in the Introduction, this leads to the two-time distribution function \eqref{eqn:Hstat}
of a limiting height interface $\mathbold{H}(x,t)$ that starts off from a one-sided Brownian motion.
Specifically, it is the large $T$ joint distributional limit of
$$ \frac{H(c_1 x_1(t_1T)^{2/3},t_1 T) - c_2 (t_1 T)}{c_3 (t_1 T)^{1/3}}, \quad \frac{H(c_1 x_2(t_2 T)^{2/3},t_2 T) - c_2 (t_2 T)}{c_3 (t_2 T)^{1/3}}$$
for the height interface \eqref{eqn:pngheight} of the polynuclear growth model with $a_1 = \sqrt{q}$, $a_i = q$ for $i > 1$ and $b_j = 1$.
Indeed, it is not too hard to see using the formula from Theorem \ref{thm:4} that the two limiting operations $T \to \infty$ and $\lambda_1 \to 0$ commute.

This model has been studied in \cite{BR}, where it is shown that the distribution function of $\mathbold{H}(0,1)$ is $F_{\rm{GOE}}^2$
where $F_{\rm{GOE}}$ is the GOE Tracy-Widom distribution function (see \cite{TW}). The single time multi-spatial distribution function
is then derived in \cite{IS}.

In order to state the result we need to define some kernels over $L^2(\R)$. Define the kernels $\K_{<}$ and $\K_{>}$ by
\begin{align} \label{eqn:Ks}
\K_s(u,v) = & e^{\mu(v-u)} \frac{1}{(2 \pi \mathbold{i})^4} \oint \limits_{\Re(\zeta) = -d_1} d\zeta \oint \limits_{\Re(\omega) = -d_2} d \omega \oint \limits_{\Re(z) = D_1} dz \oint \limits_{\Re(w) = D_2} dw \\ \nonumber
& \frac{\G(z \vt t_1, x_1, \xi_1) \G(w \vt \D t, \D x, \D \xi)}{\G(\zeta \vt t_1, x_1, \xi_1) \G(\omega \vt \D t, \D x, \D \xi)} \cdot \frac{e^{-\zeta u + \omega v} \, \zeta}{(z-\zeta)(w-\omega)(z-w)z}.
\end{align}
The condition of the contours is that $D_1 < D_2$ if $s$ equals $<$ and $D_1 > D_2$ if $s$ is $>$. All the $d$s and $D$s are positive.
The conjugation constant $\mu$ is assumed to be sufficiently large. Define also
\begin{align} \label{eqn:K10}
& \K_{1}(u,v) = e^{\mu(v-u)}  \ind{v < 0}\, \frac{1}{(2 \pi \mathbold{i})^2} \oint \limits_{\Re(\zeta) = -d} d\zeta \oint \limits_{\Re(z) = D} dz\,
\frac{\G(z \vt t_1, x_1, \xi_1)}{\G(\zeta \vt t_1, x_1, \xi_1)} \frac{e^{-\zeta u + zv}\, \zeta}{(z-\zeta) z}.
\end{align}

Define the following functions $a, b,c,d$ for $y \in \R$ and $\mu > 0$.
\begin{align*}
a(y) &= e^{-\mu y} \, \frac{1}{2 \pi \mathbold{i}} \oint \limits_{\Re(\zeta) = -d} d\zeta \, \frac{e^{-\zeta y}}{\G(\zeta \vt t_1, x_1, \xi_1)}\, , \\
b(y) &= e^{\mu y} \, \ind{y < 0} \, , \\
c(y) &= e^{\mu y} \, \frac{1}{(2 \pi \mathbold{i})^2} \oint \limits_{\Re(\omega)=-d} d \omega \oint \limits_{\Re(w)=D} dw\, 
\frac{\G(w \vt \D t, \D x, \D \xi)}{\G(\omega \vt \D t, \D x, \D \xi)} \frac{e^{\omega y}}{(w-\omega)w} \, ,\\
d(y) & = e^{\mu y} \, \frac{1}{2 \pi \mathbold{i}} \oint \limits_{\Re(\omega)=-d} d \omega \, \frac{e^{\omega y}}{\G(\omega \vt \D t, \D x, \D \xi) \omega}\, .
\end{align*}
When $\mu$ is sufficiently large these functions are bounded in absolute value by $(const) \times e^{\mu' y}\Ai(y)$. This implies the following
rank 1 kernels are of trace class.
\begin{equation} \label{eqn:Krankone}
\K_{ab}(u,v) = a(u)b(v) \quad \K_{ac}(u,v) = a(u)c(v) \quad \K_{ad}(u,v) = a(u)d(v).
\end{equation}

\begin{thm} \label{thm:6}
Consider the kernel $F_{\lambda}(\theta)$ from Theorem \ref{thm:5} in the case $r=1$ and $\lambda_1 = \lambda > 0$.
Recall the kernels $\J$ from \eqref{eqn:Js} as well. The following limits hold in the trace norm as $\lambda \to 0$:
$$ \J_{1} \to \K_{1} - \K_{ab}, \quad \J_{3,<} \to \K_{<} + \K_{ac}, \quad \J_{3,>} \to \K_{>} +\K_{ac} - \K_{ad}.$$
The kernel $\J_{2}$ does not depend on $\lambda$ and remains as is. Consequently, as $\lambda \to 0$,
\begin{align*}
F_{1,\lambda} \to \K_1 & = \J_{2} - \K_{1} + \K_{<} + \K_{ab} + \K_{ac} \\
F_{2,\lambda} \to \K_2 &= \K_{1} - \J_{2} - \K_{>} - \K_{ab} - \K_{ac} + \K_{ad}.
\end{align*}
The distribution function \eqref{eqn:Hstat} is then given by
$$ \frac{1}{2 \pi \mathbold{i}} \oint \limits_{|\theta|=r} d\theta\, \frac{\dt{I + \theta^{\ind{u > 0}} \K_1 + \theta^{-\ind{u \leq 0}} \K_2 }_{L^2(\R)}}{\theta - 1}.$$
\end{thm}

\begin{proof}
	It is enough to derive the point-wise limits of the kernels $\J$. The convergence in trace norm holds because the parameter $\lambda$
	does not affect the kernel variables $u$ and $v$, and one has that $\J = \K + \lambda \K'$ for another trace class kernel $\K'$. In
	the following we will derive the point-wise limit of each of the $\J$s separately.
	\smallskip
	
	\paragraph{\textbf{Limit of $\J_{1}$}.}
	In the formula for $\J_{1}(u,v)$ from \eqref{eqn:Js}, push the $z$-contour to the right of $\lambda$. Doing so creates
	a residue at $z = \lambda$, and one finds that $\J_{1}(u,v) = (I) - (II)$ where
	$$ (I) = e^{\mu(v-u)} \ind{v < 0} \, \frac{1}{(2 \pi \mathbold{i})^2} \oint \limits_{\Re(\zeta) = -d} d\zeta \oint \limits_{\Re(z) = D} dz\,
	\frac{\G(z \vt t_1, x_1, \xi_1)}{\G(\zeta \vt t_1, x_1, \xi_1)} \frac{e^{-\zeta u + zv}(\lambda-\zeta)}{(z-\zeta)(\lambda-z)}$$
	with $D > \lambda$, and
	$$ (II) = e^{\mu(v-u)} \ind{v<0}\, \frac{1}{2 \pi \mathbold{i}} \oint \limits_{\Re(\zeta)=-d} d \zeta \,
	\frac{\G(\lambda \vt t_1,x_1, \xi_1)}{\G(\zeta \vt t_1, x_2, \xi_1)} e^{-\zeta u + \lambda v}.$$
	In the limit $\lambda \to 0$, term $(I)$ tends to $\K_{1}(u,v)$ and term $(II)$ tends to $\K_{ab}(u,v)$.
	\smallskip
	
	\paragraph{\textbf{Limit of $\J_{3, <}$}.}
	In the formula for $\J_{3,<}$ from \eqref{eqn:Js}, first move the $w$-contour to the right of $\lambda$ and then
	the $z$-contour. Moving the $w$-contour encounters no poles, but the $z$-contour does at $z = \lambda$.
	The residue there shows that $\J_{3,<}(u,v) = (I) + (II)$ where
	$$ (I) = \text{same as}\; J_{3,<}(u,v)\; \text{but with the condition that}\; \lambda < D_1 < D_2,$$
	and
	\begin{align*}
	(II) = & e^{\mu(v-u)} \, \frac{1}{(2 \pi \mathbold{i})^3} \oint \limits_{\Re(\zeta) = -d_1} d\zeta \oint \limits_{\Re(\omega) = -d_2} d \omega \oint \limits_{\Re(w) = D_2} dw \\
	& \frac{\G(\lambda \vt t_1, x_1, \xi_1) \G(w \vt \D t, \D x, \D \xi)}{\G(\zeta \vt t_1,x_1, \xi_1) \G(\omega \vt \D t, \D x, \D \xi)} \frac{e^{-\zeta u + \omega v}}{(w-\omega)(w-\lambda)}
	\end{align*}
	again with $\lambda < D_2$.
	
	In the limit $\lambda \to 0$, term $(I)$ tends to $\K_{<}(u,v)$ and term $(II)$ tends to $\K_{ac}(u,v)$.
	\smallskip
	
	\paragraph{\textbf{Limit of $\J_{3, >}$}.}
	In the formula for $\J_{3,>}(u,v)$ one has $D_2 < D_1 < \lambda$. So, first move the $z$-contour to the
	right of $\lambda$ so that $D_1 > \lambda $ afterwards. This picks up a residue at $z = \lambda$, and
	shows that $\J_{3,>}(u,v) = (I) + (II)$ where
	$$ (I) = \text{same as}\; \J_{3,>}(u,v)\; \text{from}\; \eqref{eqn:Js}\; \text{but with}\; D_2 < \lambda < D_1$$
	and
	$$(II)  = \text{same as term} (II) \; \text{above but with the condition}\; D_2 < \lambda.$$
	
	In term $(I)$, one can move the $w$-contour to the right of $\lambda$ without encountering any poles.
	Then taking the limit $\lambda \to 0$ shows that this term converges to $\K_{>}(u,v)$.
	
	Now consider term $(II)$. Move the $w$-contour to the right of $\lambda$ with a residue at $w = \lambda$.
	This shows that $(II) = (III) + (IV)$ where
	$$(III) = \text{same as}\; (II)\; \text{but with the condition that}\; \lambda < D_2$$
	and
	$$(IV) = e^{\mu(v-u)} \oint \limits_{\Re(\zeta) = -d_1} d\zeta \oint \limits_{\Re(\omega) = -d_2} d \omega
	\frac{\G(\lambda \vt t_1, x_1, \xi_1) \G(\lambda \vt \D t, \D x, \D \xi)}{\G(\zeta \vt t_1, x_1, \xi_1) \G(\omega \vt \D t, \D x, \D \xi)} \frac{e^{-\zeta u + \omega v}}{\lambda - \omega}.$$
	
	In the limit $\lambda \to 0$, term $(III)$ tends to $\K_{ac}(u,v)$ and term $(IV)$ tends to $-\K_{ad}(u,v)$.
\end{proof}

\section*{Acknowledgements}
We thank a referee for helpful comments, in particular, leading to a better Theorem \ref{thm:4}.

Kurt Johansson's research is partially supported by grant KAW 2015.0270 from the Knut and Alice Wallenberg Foundation
and grant 2015-04872 from the Swedish Science Research Council (VR).

\end{document}